\documentclass[a4paper,leqno,11pt]{amsart}

\usepackage[english]{babel} 									%Support for many languages
\usepackage[T1]{fontenc}							%The last two are related to making "ä" etc. more usable in Latex
\usepackage[utf8]{inputenx}
				
\usepackage{mathtools}								%Includes amsmath, basic math commands
\usepackage{empheq}									%Useful for systems of equations
\usepackage{enumitem}								%Improves enumerate environment among other things

\usepackage{amssymb}								%Fonts and mathematical symbols
\usepackage{mathrsfs}								%\mathscr command; http://www.math.washington.edu/~lee/Courses/typesetting-script.pdf
\usepackage{bm}                                     % Bold all the things in math mode 
				
\usepackage{verbatim}                               %comment environment: hide everything
\usepackage{iflang}									%Used in theoremenv.sty
\usepackage{xifthen}								%See http://tex.stackexchange.com/questions/58628/optional-argument-for-newcommand
\usepackage[final]{pdfpages}						%Enables including multiple pages of a pdf easily.
\usepackage{todonotes}								%\todo{note} \usepackage[disable]{todonotes}
\usepackage{aliascnt}								%Useful for making a counter share the number with other counter, but with different names.. used in thereomenv.sty in amsthm 																								 variants of the code
\usepackage{needspace}								%Needed in theoremenv, checks that proofs etc. don't get separated too easily: http://tex.stackexchange.com/a/30125
\usepackage{refcount}

\usepackage{subcaption}			%Naming figures (subfigure)
\usepackage{placeins} 			%Restricting movement of floats (FloatBarrier)
\usepackage{grffile}			%More possibilities with naming figures
\usepackage[all]{xy}			%Useful in Algebraic Topology

\usepackage{transparent}
\usepackage{color}				%Colorful pictures

\usepackage{tikz}               % For drawing diagrams.
\usetikzlibrary{math,arrows.meta, cd}
\DeclareGraphicsExtensions{.pdf,.png,.jpg}

\usepackage{graphicx}			%More options for includegraphic. See http://ctan.org/pkg/graphicx for example.
 
\usepackage{diagbox} %Diagonals to tableau

\usepackage[babel]{csquotes}							%More options for \cite etc.					

\usepackage{nameref}									%I.e. can refer to Theorem (Euler's theorem) as "Euler's theorem" with a hyperlink
\usepackage[colorlinks	=	true,						%Hyperlinks to references (an absolute must have)
          	 	linkcolor	=	red,
            	urlcolor	=	red,
            	citecolor	=	red]%
            	{hyperref}
\usepackage{bookmark}									%Hyperref options; used in theoremenv to bookmark theorems in pdf files automatically
\usepackage{cleveref}									%Further options for hyperref, \cref and \Cref are interesting (and their multiple variants)
\usepackage{thmtools}
\usepackage{thm-restate}

\numberwithin{equation}{section}

\newtheorem{theorem}{Theorem}[section]

\newtheorem{corollary}[theorem]{Corollary}
\newtheorem{lemma}[theorem]{Lemma}
\newtheorem{proposition}[theorem]{Proposition}

\newtheorem*{theorem*}{Theorem}

{\bf}{\it}

\newtheorem{definition}[theorem]{Definition}
\newtheorem*{definition*}{Definition}

\newtheorem*{question*}{Question}

\theoremstyle{remark}
\newtheorem{remark}[theorem]{Remark}

\newtheorem*{remark*}{Remark}

\theoremstyle{remark}

\renewcommand{\S}{\mathbb{S}}

\newcommand{\R}{\mathbb{R}}

\newcommand{\cC}{\mathcal{C}}

\DeclareMathOperator{\diam}{diam}
\DeclareMathOperator{\dist}{dist}
\DeclareMathOperator{\loc}{loc}

\allowdisplaybreaks

\newcommand{\norm}[1]{ \left\Vert #1 \right\Vert }	%Norm
\newcommand{\apmd}[2][]{							%Approximate Metric Derivative with options
	\ifthenelse{\equal{#1}{}}%
					{ \operatorname{N}_{#2}	}%
					{ \operatorname{N}_{#1,#2} 	}}

\newcommand{\aint}[2][]{%       %Integral average ($\aint[a]{b} corresponds to $\int_{a}^{b}$ with "bar" and $\aint{X}$ to $\int_{X}$ with bar.
	\ifthenelse{\equal{#1}{}}%
					{%
\mathchoice%
      {\mathop{\kern 0.2em\vrule width 0.6em height 0.69678ex depth -0.58065ex
              \kern -0.8em \intop}\nolimits_{\kern -0.45em#2}^{#1}}%
      {\mathop{\kern 0.1em\vrule width 0.5em height 0.69678ex depth -0.60387ex
              \kern -0.6em \intop}\nolimits_{#2}^{#1}}%
      {\mathop{\kern 0.1em\vrule width 0.5em height 0.69678ex depth -0.60387ex
              \kern -0.6em \intop}\nolimits_{#2}^{#1}}%
      {\mathop{\kern 0.1em\vrule width 0.5em height 0.69678ex depth -0.60387ex
              \kern -0.6em \intop}\nolimits_{#2}^{#1}}}%
					{%
\mathchoice%
      {\mathop{\kern 0.2em\vrule width 0.6em height 0.69678ex depth -0.58065ex                                              
              \kern -0.8em \intop}\nolimits_{\kern -0.45em#1}^{#2}}%
      {\mathop{\kern 0.1em\vrule width 0.5em height 0.69678ex depth -0.60387ex
              \kern -0.6em \intop}\nolimits_{#1}^{#2}}%
      {\mathop{\kern 0.1em\vrule width 0.5em height 0.69678ex depth -0.60387ex
              \kern -0.6em \intop}\nolimits_{#1}^{#2}}%
      {\mathop{\kern 0.1em\vrule width 0.5em height 0.69678ex depth -0.60387ex
              \kern -0.6em \intop}\nolimits_{#1}^{#2}}}}

\newcommand{\vol}{\mathrm{vol}}

\newcommand{\C}{\mathbb{C}}

\newcommand{\bS}{\mathbb{S}}

\newcommand{\comass}{\mathrm{comass}}
\newcommand{\symp}{\mathrm{symp}}
\newcommand{\Gr}{\mathrm{Gr}}

\newcommand{\SL}{\mathrm{SL}}
\newcommand{\SO}{\mathrm{SO}}
\newcommand{\CO}{\mathrm{CO}}

\newcommand{\bT}{\mathbb{T}}

\newcommand{\assoc}{\mathrm{assoc}}
\newcommand{\coassoc}{\mathrm{coassoc}}
\newcommand{\Cayley}{\mathrm{Cayley}}

\begin{document}
\title{Liouville's theorem in calibrated geometries}
%\title[Liouville and conformal rigidness]{Liouville's theorem in low dimensional calibrated geometries}

\author{Toni Ikonen} 
\address{Department of Mathematics, University of Fribourg, Chemin du Musée 23, 1700 Fribourg, Switzerland.}

\email{toni.ikonen@unifr.ch}

\author{Pekka Pankka} 
\address{Department of Mathematics and Statistics, P. O. Box 68 (Pietari Kalmin katu 5), FI-00014 University of Helsinki, Finland}

\email{pekka.pankka@helsinki.fi}

\thanks{Both authors were supported by the Academy of Finland, project number 332671. The first author was also supported by the Swiss National Science Foundation grant 212867.}

\begin{abstract}
We consider the following extension of the classical Liouville theorem: A calibration $\omega \in \Lambda^n \mathbb{R}^m$, where $3 \le n \le m$, has the \emph{Liouville property} if a Sobolev mapping $F\colon \Omega \to \mathbb{R}^m$, where $\Omega \subset \mathbb{R}^n$ is a domain, in $W^{1,n}_{\loc}( \Omega, \mathbb{R}^m )$ satisfying $\norm{DF}^n = \star F^{*}\omega$ almost everywhere is a restriction of a M\"obius transformation $\mathbb{S}^m \to \mathbb{S}^m$. 

We show that, for $m\ge 5$, every calibration in $\Lambda^{m-2} \mathbb{R}^m$ has the Liouville property and, in low dimensions, a calibration $\omega \in \Lambda^n \mathbb{R}^m$ has the Liouville property for $3 \le n \le m \le 6$ unless $\omega$ is face equivalent to the Special Lagrangian. In these cases, the Liouville property stems from isoperimetric rigidity of these mappings together with a classification of calibrations whose conformally flat calibrated submanifolds are flat.

We also show that, for $3 \leq n \leq m$, the calibrations with the Liouville property form a dense $G_\delta$ set in the space of calibrations. As an application, we consider factorization of more general quasiregular curves and stability of quasiregular curves of small distortion.
\end{abstract}

\maketitle\thispagestyle{empty}

{
    \hypersetup{linkcolor=black}
    \setcounter{tocdepth}{1}
    \tableofcontents
}

\section{Introduction}\label{sec:intro}

The classical Liouville theorem for conformal mappings is a starting point for rigidity theory in higher dimensional conformal geometry: \emph{a (nonconstant) conformal mapping $f\colon \Omega \to \R^n$, where $\Omega$ is a domain in $\R^n$ for $n\ge 3$, is a restriction of a M\"obius transformation $\bS^n \to \bS^n$}, that is, after identification $\bS^{n}= \widehat{\R}^n \coloneqq \R^n \cup \left\{\infty\right\}$, via stereographic projection, the mapping $f$ is given by formula 
\begin{equation}\label{eq:Mobius}
    f(x) = y_0 + L\left( \frac{ x - x_0 }{ |x-x_0|^{\epsilon} } \right)
    \quad\text{for $x \in \R^n\setminus \{x_0\}$},
\end{equation}
where $x_0, y_0 \in \bS^n$, $\epsilon \in \left\{0,2\right\}$, and $L\in \CO(n)$ is an orientation preserving conformal linear map. Recall that $A \in \CO(n)$ if $A = \lambda O$, where $O \in \SO(n)$ and $\lambda \ge 0$.

In the aforementioned statement, a mapping $f\colon \Omega \to \R^n$ is \emph{conformal} if $f$ is smooth in the classical sense and \emph{weakly conformal}, that is, $(Df)_x \in \CO(n)$ for each $x\in \Omega$. In particular, the Liouville's theorem states that \emph{for $n\ge 3$, modulo a Möbius transformation in the domain, the differential $Df$ of a conformal map $f \colon \Omega \to \R^n$ is constant}. We refer to the monographs of Iwaniec and Martin \cite[Chapter 5]{Iw:Ma:01} and Gehring, Palka, and Martin \cite[Chapter 3]{Geh:Mar:Pal:17} for a detailed discussion on the classical Liouville's theorem.

In the 1960's Gehring \cite{Geh:61} and Reshetnyak \cite{Res:60,Res:67} showed that, to obtain this rigidity result, it suffices to assume Sobolev $W^{1,n}_{\loc}$-regularity from the map $f$. More precisely, \emph{if $f\colon \Omega \to \bS^n$, where $\Omega \subset \R^n$ is a domain and $n\ge 3$, is in the Sobolev space $W^{1,n}_{\loc}(\Omega, \bS^n)$ and $(Df)_x \in \CO(n)$ for almost every $x\in \Omega$, then $f$ is a restriction of a M\"obius transformation $\bS^n \to \bS^n$.}

The Sobolev regularity in theorems of Gehring and Reshetnyak was later relaxed to $W^{1,p}_{\loc}$-regularity for $p=n/2$ in even dimensions by Iwaniec and Martin \cite{Iwa:Mar:93} and in odd dimensions to $W^{1,p}_{\loc}$-regularity for $p=p(n)<n$ by Iwaniec \cite{Iwa:92}. We refer to \cite{Iw:Ma:01} on open questions on the Sobolev regularity in this theorem.

\begin{remark*}
We note in passing that, although no local injectivity on the map $f \colon \Omega \to \bS^n$ is assumed, the non-negativity of the Jacobian implies that the mapping $f$ is topologically weakly orientation preserving and has \emph{a posteriori} discrete fibers. Due to the low analytic regularity of the mapping, this additional information on the orientation of the map is crucial in the result. For example, the folding map $\R^3 \to \R^3$, $(x_1,x_2,x_3) \mapsto (|x_1|, x_2,x_3)$, is $1$-Lipschitz and its differential is almost everywhere an orthogonal matrix.
\end{remark*}

In this article, we consider versions of the Liouville's theorem in the context of \emph{calibrated geometry}, that is, rigidity of conformal mappings $F \colon \Omega \to \R^m$, where $\Omega \subset \R^n$ and $2\le n < m$ and the differential $DF$ belongs to a conformal matrices associated to a constant coefficient calibration $\omega\in \Lambda^n \R^m$; here and in what follows we identify covectors $\Lambda^n \R^m$ with constant coefficient smooth forms in $\Omega^n(\R^m)$. Recall that a covector $\omega\in \Lambda^n \R^m$ is a \emph{calibration} if
\[
\norm{\omega}_\comass := \sup\{ \omega(v_1,\ldots,v_n) \colon v_i \in \R^m,\ |v_1|=\cdots=|v_n| = 1\} = 1.
\]
We refer to the seminal paper of Harvey and Lawson \cite{Har:Law:82} and the monograph of Joyce \cite{Joy:07} for the theory of calibrations and merely note at this point that the standard volume form $\vol_{\R^n}\in \Lambda^n \R^n$ in $\R^n$ is a calibration.

Given a calibration $\omega \in \Lambda^n \R^m$, we denote
\begin{align*}
\SO(\omega) &= \{ O \colon \R^n \to \R^m \colon O \text{ linear isometry},\ O^*\omega = \vol_{\R^n} \},\\
\CO(\omega) &= \{ \lambda O \colon \R^n \to \R^m \colon O \in \SO(\omega),\ \lambda \ge 0\},
\end{align*}
and
\[
\Gr(\omega) = \{ V \subset \R^m \colon V \text{ linear subspace},\ \norm{\omega|_V}_{\comass} = 1 \}.
\]
Note that, for $A \in \CO(\omega)$, we have $\mathrm{Im}A = A\R^n \in \Gr(\omega)\cup \{\{ 0 \}\}$. We say that
calibrations $\omega, \omega' \in \Omega^n( \mathbb{R}^m )$ are \emph{equivalent} if $\CO(\omega) = \CO(\omega')$. Calibrations $\omega$ and $\omega'$ are \emph{face equivalent} if there exists $O \in \SO(m)$ for which $\Gr(\omega) = \Gr(O^*\omega')$; see e.g.~Joyce \cite[Section 4.3]{Joy:07}. 

\begin{definition*}
A Sobolev $W^{1,n}_{\loc}$-mapping $F\colon \Omega \to \R^m$, where $\Omega \subset \R^n$ is an open set, is a \emph{conformal $\omega$-curve for a calibration $\omega \in \Lambda^n(\R^m)$} if $(DF)_x \in \CO(\omega)$ for almost every $x\in \Omega$.
\end{definition*}

Since $\SO(\vol_{\R^n}) = \SO(n)$ and  $\CO(\vol_{\R^n}) = \CO(n)$, we readily have that, for a domain $\Omega \subset \R^n$, a map $\Omega \to \R^n$ is conformal if and only if it is a conformal $\vol_{\R^n}$-curve. Thus this definition agrees with the definition of conformal mappings above.

In addition to conformal maps between equidimensional spaces, for example, pseudoholomorphic curves associated to symplectic forms, see e.g. \cite{Grom:85} or \cite[Chapter 2]{McD:Sal:04}, and Smith maps associated to vector cross products are conformal curves. We refer to Cheng, Karigiannis, and Madnick \cite{Che:Kari:Mad:20} and \cite{Che:Kar:Mad:23} for discussion on Smith maps.

\begin{remark*}
We may also characterize conformal $\omega$-curves $F\colon \Omega \to \R^m$ for a calibration $\omega \in \Lambda^n(\R^m)$ as Sobolev mappings in $W^{1,n}_{\loc}(\Omega, \R^m)$ satisfying $\norm{DF}^n = \star F^*\omega$ almost everywhere in $\Omega$; see Section \ref{sec:preliminaries}. Note that, by \cite[Corollary 1.7]{Iko:24}, there exists $p = p(n,m) \in ( n/2, n )$ for which every $f \in W^{1,p}_{loc}( \Omega, \mathbb{R}^m )$ satisfying $(DF)_x \in \CO(\omega)$ for almost every $x\in \Omega$ is, in fact, a conformal $\omega$-curve.
\end{remark*}

We may now restate Liouville's theorem as follows: 
\emph{for $n\ge 3$, nonconstant conformal $\vol_{\R^n}$-curves $\Omega \to \R^n$, where $\Omega \subset \R^n$ is a domain, are restrictions of M\"obius transformations $\bS^n \to \bS^n$.} 
The Liouville problem for constant coefficient calibrations reads as follows.

\begin{question*}[Liouville problem for curves]
For $3 \le n \le m$ and a constant coefficient calibration $\omega \in \Lambda^n \R^m$, are nonconstant conformal $\omega$-curves $\Omega \to \R^m$, where $\Omega\subset \R^m$ is a domain, restrictions of M\"obius transformations $\bS^m \to \bS^m$? In case the answer is positive for a calibration $\omega \in \Lambda^n \R^n$, we say that $\omega$ has the \emph{Liouville property}.
\end{question*}
Here and in what follows, we identify $\R^n$ with $\R^n \times \left\{0\right\} \subset \R^m$ and $\R^m \cup \left\{\infty\right\}$ with $\S^m$ using the stereographic projection. This allows us to consider $\Omega$ as a subset of $\S^m$ and thus the restriction question above is well-posed.

\subsection{Restating the Liouville problem}

Before considering this Liouville problem for particular calibrations, we reformulate the Liouville problem using two methods. The first is the classical regularity which follows from classical elliptic regularity theory (see e.g.~Uhlenbeck \cite{Uhl:77} and Hardt--Lin \cite{Har:Lin:87}), minimal surface theory (see e.g.~Harvey--Lawson \cite{Har:Law:82} and Reifenberg \cite{Reif:64}), and conformal mapping theory (see e.g. Lassas--Liimatainen--Salo \cite{Las:Lii:Sal:22}); see also Iliashenko and Karigiannis \cite{Ilia:Karig:23} for a related $\mathcal{C}^{1,\alpha}$-regularity result.

\begin{theorem}\label{thm:realanalytic}
Let $\omega \in \Lambda^n \R^m$ be a calibration, for $3\le n \le m$, and let $F\colon \Omega \to \R^m$ be a conformal $\omega$-curve, where $\Omega \subset \R^n$ is a domain. Then $F$ is $\mathcal{C}^{1,\alpha}$-regular and the norm function $\norm{DF}^{\frac{n-2}{2}}$ is subharmonic. Furthermore, $F$ is a real analytic immersion in $\Omega_F = \{\norm{DF}>0\}$ and the image $F(\Omega_F)$ is an $\omega$-calibrated and real-analytic conformally flat immersed submanifold of $\R^m$.
\end{theorem}

The other key ingredient in reformulating the Liouville problem is the affine rigidity of conformal $\omega$-curves, which stems from a Reshetnyak type isoperimetric inequality for conformal $\omega$-curves; see Theorem \ref{thm:isoperimetric-inequality}. This isoperimetric inequality yields the following characterization of affine curves among conformal $\omega$-curves.
 
\begin{restatable}{theorem}{constantNorm}
\label{thm:constant-norm-means-affine}
Let $2\le n \le m$, $\omega \in \Lambda^n \R^m$ be a calibration, $\Omega \subset \R^n$ be a domain, and let $F\colon \Omega \to \R^m$ be a conformal $\omega$-curve. Then $F$ is affine if and only if $\norm{DF}$ is constant.
\end{restatable}

\begin{remark*}
We note in passing that \Cref{thm:constant-norm-means-affine}, together with the subharmonicity of $\|DF\|^{ \frac{n-2}{2} }$ when $n \geq 3$ and $\log \|DF\|$ when $n = 2$, yields a maximum principle for conformal $\omega$-curves; see \Cref{cor:Maximum-principle}.
\end{remark*}

Theorems \ref{thm:realanalytic} and \ref{thm:constant-norm-means-affine}, together with characterization of the norm functions for the differential of conformal $\omega$-curves, yield now the following reformulation of the Liouville problem in terms of inner M\"obius curves.

\begin{definition*}
A mapping $F \colon \Omega \to \R^m$, where $\Omega \subset \R^n$ is open, is an \emph{inner Möbius curve} if
\begin{equation*}
    F = A \circ g,
\end{equation*}
where $A \colon \R^n \to \R^m$ is an affine map $x \mapsto y_0 + L x$ for $y_0 \in \R^m$ and an orthogonal linear map $L \colon \R^n \to \R^m$, and $g$ is a restriction of a (nonconstant) Möbius transformation $M \colon \S^n \to \S^n$. A calibration $\omega \in \Lambda^n \R^m$ is \emph{inner M\"obius rigid} if every nonconstant conformal $\omega$-curve is an inner M\"obius curve.
\end{definition*}

\begin{remark*}
The inner M\"obius property of a conformal $\omega$-curve is local in the following sense: \emph{a conformal $\omega$-curve $F \colon \Omega \to \R^m$ that is an inner M\"obius curve in an open set $\emptyset \neq U \subset \Omega$ of a domain $\Omega$ is an inner M\"obius curve}; see \Cref{cor:comparison-principle} for a precise statement.
\end{remark*}

\begin{restatable}{theorem}{innermob}
\label{thm:inner-Mobius-Liouville:rigid}
Let $\omega \in \Lambda^n \R^m$ be a calibration, for $3 \leq n \leq m$, and let $F \colon \Omega \to \R^m$ be a nonconstant conformal $\omega$-curve in a domain $\Omega \subset \R^n$. Then $F$ is the restriction of a Möbius transformation $\S^m \to \S^m$ if and only if $F$ is an inner Möbius curve.
\end{restatable}

In light of \Cref{thm:inner-Mobius-Liouville:rigid}, a calibration has the Liouville property if and only if it is inner Möbius rigid. \emph{A fortiori}, Theorems \ref{thm:realanalytic} and \ref{thm:inner-Mobius-Liouville:rigid} translate the question on inner M\"obius rigidity to a question on the existence of non-affine conformally flat $\omega$-calibrated submanifolds in $\R^m$.

\subsection{Inner M\"obius rigidity}

Calibrations in $\Lambda^2 \R^m$ are not M\"obius rigid. Indeed, holomorphic and pseudoholomorphic curves give examples of curves which are not restrictions of M\"obius transformations. In codimension two, the situation is drastically different.

\begin{restatable}{theorem}{codimMobius}
\label{thm:codim-2}
For $m\ge 5$, every calibration in $\Lambda^{m-2} \R^m$ is inner M\"obius rigid.
\end{restatable}

Recall that, if $\omega \in \Lambda^{m-2} \R^m$ is a calibration, so is $\star \omega \in \Lambda^{2}\R^m$ and the Grassmannian $\Gr(\omega)$ is mapped isomorphically to $\Gr(\star \omega)$ by the Hodge star operator. Therefore the inner M\"obius rigidity in \Cref{thm:codim-2} is not related to the dimension of the Grassmannian $\Gr(\omega)$.

The first example, in higher dimensions, of a calibration without the Liouville property is the Special Lagrangian $\omega_\SL = \Re(dz_1\wedge \cdots \wedge dz_n) \in \Lambda^n \C^{n}$, where we identify $\mathbb{C}^n$ with $\mathbb{R}^{2n} = \R^n \times \R^n$, for $n \geq 2$. For the purpose of our discussion, we recall two examples of conformal $\omega_{\SL}$-curves. Our first example of a non-trivial conformal $\omega_\SL$-curve is given by the Lagrangian catenoid 
\[
    \mathcal{M}_\SL
    \coloneqq
    \left\{
        (x,y) \in \R^{2n}
        \mid
        |x|y = |y|x,\,
        \mathrm{Im}( |x|+i|y| )^n = 1,\,
        |y| < |x| \tan\left( \frac{\pi}{n} \right)
    \right\}
\]
which is calibrated by the Special Lagrangian; see \cite[Theorem 3.5]{Har:Law:82}. The following example stems from the works of Castro and Urbano \cite{Cas:Urb:99} and Blair \cite{Bla:07} on conformal parametrizations of the Lagrangian catenoid.

\begin{proposition}
\label{prop:SL-catenoid}
Let $\widehat F \colon \R\times \bS^{n-1} \to \R^{2n}$ be the map $(t,p) \mapsto f(t) \iota(p)$, where $\iota \colon \R^n \to \C^n$ is the embedding $p \mapsto (p,0)$, and $f \colon \R\to \C$, $t\mapsto (\sinh nt - i)^{1/n}$. Then the map
\[
F = \widehat F \circ \varphi \colon \R^n\setminus \{0\} \to \R^{2n},
\]
where $\varphi \colon \R^n\setminus \{0\} \to \R \times \bS^{n-1}$ is the map $x\mapsto \left( \log |x|, \frac{x}{|x|} \right)$, is a conformal $\omega_{\SL}$-curve onto $\mathcal{M}_\SL$.
\end{proposition}

\begin{remark*}
Note that, for $n\ge 3$, the singularity at the origin is essential in the sense that there is no conformal embedding $\R^n \to \R^{2n}$ having the same image. Indeed, $F$ is an embedding and hence $F(\R^n\setminus \{0\})$ is homeomorphic to $\R\times \bS^{n-1}$. However, by Zorich's theorem (see e.g.~Bonk and Heinonen \cite[Proposition 1.4]{Bo:Hei:01}), every nonconstant conformal map $\R^n \to F(\R^n\setminus \{0\})$ is a homeomorphism.
\end{remark*}

Our second example is an entire, but non-affine, calibrated $\omega_\SL$-curve $\R^n \to \R^{2n}$ based on the Clifford torus. The following map is based on the Anti\'c and Vranken classification \cite{Anti:Vra:22} of conformally flat Special Lagrangian submanifolds having Schouten tensor of exactly one eigenvalue of multiplicity one.
\begin{proposition}
\label{prop:SL-example-concrete}
Let $n \geq 2$ and let $h \colon \R^n \times \R \to \C^n$ be the mapping
\[
(y_1,\ldots, y_n,t) \mapsto \frac{i^{-(n-1)}}{\sqrt{n}} e^t \left( e^{i\sqrt{n} y_1}, \ldots, e^{i\sqrt{n} y_n}\right).
\]
Then there exists a linear isometry $O \colon \R^n \to \R^{n+1}$ for which the mapping 
\[
F = h \circ O \colon \R^n \to \R^{2n}
\]
is a conformal $\omega_\SL$-curve. 
\end{proposition}

\begin{remark*}
For $n \geq 3$, the image of $F \colon \R^n \to \R^{2n}$ 
in \Cref{prop:SL-example-concrete} is a punctured cone over the link $\mathbb{T} = N \cap \mathbb{S}^{2n-1}$, where $\mathbb{T}$ is conformally equivalent to an $(n-1)$-dimensional torus. In dimension $n=3$, this is in contrast to cones over higher genus surfaces. In \cite{Has:Kap:07}, Haskins and Kapouleas showed that, for every odd genus $g \geq 1$, there exists a countably infinite family of Special Lagrangian $3$-subvarieties $N_g$ which are cones over genus $g$ surfaces $\Sigma_g = N_g \cap \mathbb{S}^{5}$. For $g \ge 3$, the punctured cone $N_g\setminus \{0\}$ is not an image of a conformal $\omega_{\SL}$-curve $\R^3 \to \R^m$ by a version of Varopoulos's theorem \cite[Theorem X.5.1]{Varopoulos:91} in the context of conformal $\omega$-curves. We discuss this in more detail in Section \ref{sec:conical-singularities}.
\end{remark*}

Our interest to state these concrete examples of conformal $\omega_\SL$-curves stems from the observation that $\omega_\SL$ is the only non-rigid calibration in $\Lambda^n \R^m$ for $3 \le n \le m \le 6$. The following theorem extends Theorem \ref{thm:codim-2} to the remaining cases in these dimensions.

\begin{restatable}{theorem}{classification}
\label{thm:classification-dim-six}
Let $\omega \in \Lambda^n \R^m$ be a calibration, where $3\le n \le m\le 6$, which is not face equivalent to $\omega_{\SL}$, let $\Omega \subset \R^n$ be a domain, and let $F \colon \Omega \to \R^m$ be a conformal $\omega$-curve. Then $F$ is an inner M\"obius curve. In particular, $\omega$ is inner M\"obius rigid.
\end{restatable}

The proof of \Cref{thm:classification-dim-six} has two main ingredients. First is the classification of calibrations in $\R^6$ by Dadok and Harvey \cite{Dad:Har:83} and Morgan \cite{Mor:85}; see also Joyce \cite[Section 4.3]{Joy:07}. The second main ingredient in the proof of \Cref{thm:classification-dim-six} is the fact that if a calibration $\omega \in \Lambda^n \R^m$, for $3 \leq n \leq m \leq 6$, is not face equivalent to the Special Lagrangian, then conformally flat $\omega$-calibrated submanifolds are flat submanifolds. Since the submanifolds in question are minimal surfaces, they are \emph{a fortiori} affine by minimal surface theory and, therefore, $\omega$ is inner M\"obius rigid.

Regarding dimensions $m=7$ and $m=8$, Proposition \ref{prop:SL-example-concrete} implies that the associative calibration $\omega_\assoc \in \Lambda^3 \R^7$ and the Cayley calibration $\omega_\Cayley\in \Lambda^4 \R^8$ are not inner M\"obius rigid. We do not know if the coassociative form $\omega_\coassoc = \star \omega_\assoc \in \Lambda^4 \R^7$ is inner M\"obius rigid.

\subsection{Applications to quasiregular curves}

As an application, we consider rigidity results for quasiregular curves. A map $F\colon \Omega \to \R^m$, where $\Omega \subset \R^n$ is an open set, is a \emph{$K$-quasiregular $\omega$-curve for $K\ge 1$}, if $F\in W^{1,n}_{\loc}(\Omega, \R^m)$ and $\norm{DF}^n \le K (\star F^*\omega)$ almost everywhere. Here $\norm{DF}$ is the operator norm of the differential $DF$ and $\star$ the Hodge star. We refer to \cite{Pan:20} for further discussion on this definition. Beltrami systems associated to almost complex structures give examples of quasiregular curves, see Donaldson \cite{Don:96} and Iwaniec, Verchota, and Vogel \cite[Section 7]{Iwa:Ver:Vog:02}; see also \cite[Section 3]{Hei:Pan:Pry:23}. In this terminology, a conformal $\omega$-curve is a $1$-quasiregular $\omega$-curve and a $K$-quasiregular map is a $K$-quasiregular $\vol_{\R^n}$-curve. Recall also that a quasiconformal map is a quasiregular homeomorphism.

Theorem \ref{thm:constant-norm-means-affine} yields a general conditional factorization statement in spirit of Sto\"ilow's theorem. Recall that the classical Sto\"ilow's theorem states that \emph{a quasiregular mapping $f\colon \Omega \to \R^2$ is a composition $f = h \circ \varphi$, where $\varphi \colon \Omega \to \Omega'$ is quasiconformal and $h \colon \Omega' \to \C$ holomorphic}.

\begin{restatable}{theorem}{qrStoilow}
\label{thm:quasiregular-Stoilow}
Let $2\le n \le m$, $\omega \in \Lambda^n \R^m$ be a calibration, $\Omega \subset \R^n$ be a domain, and let $F\colon \Omega \to \R^m$ be a quasiregular $\omega$-curve for which $(DF)_x\R^n \in \Gr(\omega)$ and $\star F^{*}\omega(x) \geq 0$ for almost every $x\in \Omega$. Then $(DF)^t DF = (Df)^t Df$ for a quasiregular map $f \colon \Omega \to \R^n$ if and only if $F = A \circ f$, where $A\colon \R^n \to \R^m$, $x\mapsto y_0 + Lx$, for $L\in \SO(\omega)$. In particular, the image $F(\Omega)$ is a smooth (affine) $n$-manifold.
\end{restatable}

Since the differential $Df$ of a quasiregular mapping $f\colon \Omega \to \R^n$ is almost everywhere non-vanishing, the symmetric matrix field $x\mapsto (DF)_x^t (DF)_x$ in \Cref{thm:quasiregular-Stoilow} may be viewed as a measurable (possibly degenerate) Riemannian metric $g_F$ on $\Omega$. Having the Lagrangian catenoid example $F\colon \R^n\setminus \{0\} \to \C^n$ at our disposal, we may construct a measurable Riemannian metric on $\R^n$ from a quasiregular curve that is not induced by a quasiregular mapping. We formulate this as follows.

\begin{restatable}{corollary}{nonSolvability}
\label{cor:non-solvability}
There exists a quasiregular $\omega_\SL$-curve $H \colon \R^n \to \R^{2n}$ for which the equation 
\[
    (DH)^t DH = (Df)^t Df
\]
is not solvable for quasiregular mappings $f\colon \Omega \to \R^n$ in any subdomain $\Omega \subset \R^n$. Moreover, $H$ is a branched cover of its image --- the Lagrangian catenoid --- and the differential $DH$ exists and has full-rank apart from a set of Hausdorff codimension two.
\end{restatable}

\subsection{Space of inner Möbius rigid calibrations}

We finish this introduction with properties of the space of inner Möbius rigid calibrations $\mathcal I_n(\R^m) \subset \Lambda^n \R^m$ in the space of all calibrations $\mathcal{C}_n(\R^m) \subset \Lambda^n \R^m$ for $3 \leq n \leq m$. We prove the following dichotomy. For the statement, recall that a subset $E$ of a topological space is a \emph{$G_\delta$ set} if $E$ is the intersection of a countably many open sets.

\begin{restatable}{theorem}{compactness}
\label{thm:compactnessproperty}
For $3 \leq n \leq m$, either $\mathcal I_n(\R^m) = \mathcal{C}_{n}( \R^m )$ or $\mathcal{I}_n( \R^m ) \subset \mathcal{C}_n( \R^m )$ is a non-compact and dense $G_\delta$ set.
\end{restatable}

Both cases in \Cref{thm:compactnessproperty} occur. Indeed, \Cref{thm:codim-2} proves the equality $\mathcal{I}_{m-2}( \R^m ) = \mathcal{C}_{m-2}( \R^m )$ for $m \geq 5$, while \Cref{prop:SL-catenoid} implies that $\mathcal{I}_{n}( \R^{m} )$ is not compact for $m \geq 2n \geq 6$. For $m \geq 3$, the equality $\mathcal{I}_{m}( \R^m ) = \mathcal{C}_{ m }( \R^m )$ is a reformulation of the classical Liouville's theorem and, for $m \geq 4$, the equality $\mathcal{I}_{m-1}( \R^{m} ) = \mathcal{C}_{ m-1 }( \R^{m} )$ follows from the classical Liouville's theorem and the fact that every element in $\mathcal{C}_{ m-1 }( \R^{m} )$ is simple.

Our interest to the $G_\delta$ property and density of the space $\mathcal I_n(\R^m)$ in $\mathcal{C}_n(\R^m)$ stems from the observation that $\mathcal{I}_n(\R^m)$ has a neighborhood $\mathcal{U}$ in $\Lambda^n \R^m$ having the property that, for a calibration $\omega \in \mathcal U$, conformal $\omega$-curves are weakly quasisymmetric. We finish the introduction with this result. Recall that a mapping $F\colon \Omega \to \R^m$, where $\Omega \subset \R^n$ is an open set, is \emph{$(H,\rho)$-weakly quasisymmetric for $H \geq 1$ and $\rho \in (0,1/4]$} if, for every ball $B( x_0, r_0 ) \subset \Omega$,
\begin{equation}\label{eq:weakly:QS}
    \sup_{ y \in B( x_0, \rho r_0 ) } | F(y) - F(x_0) |
    \leq
    H
    \inf_{ y \in \partial B( x_0, \rho r_0 ) } | F(y) - F( x_0 ) |.
\end{equation}
It follows from \eqref{eq:weakly:QS} that an $( H, \rho )$-weakly quasisymmetric map is either constant or a topological immersion. In particular, such a map is discrete if nonconstant.

\begin{restatable}{theorem}{weakQS}\label{thm:localinjec}
Let $3 \leq n \leq m$ and $H>1$. Then there exists $\rho = \rho(H) \in ( 0, 1/4 ]$ and an open set $\mathcal{U} \subset \Lambda^n \R^m$ containing $\mathcal{I}_{n}( \R^m )$ and having the following property: for a calibration $\omega \in \mathcal U$, every conformal $\omega$-curve $\Omega\to \R^m$ is $( H, \rho )$-weakly quasisymmetric. In particular, such a curve is constant or a discrete map.
\end{restatable}

In general, quasiregular curves of large distortion need not be discrete mappings, see e.g.~Iwaniec, Verchota, and Vogel \cite{Iwa:Ver:Vog:02}, Rosay \cite{Ros:10}, or \cite[Sections 2 and 3]{Hei:Pan:Pry:23}. Recall that a mapping $F\colon \Omega \to \R^m$ is \emph{discrete} if $F^{-1}(y)$ is a discrete subset of $\Omega$ for each $y\in \R^m$.

\subsection*{Organization of the article}

This article is organized as follows. After preliminaries, we discuss in Section \ref{sec:C1} the classical smoothness of conformal curves. In Section \ref{sec:isoperimetry}, we discuss the isoperimetric inequality for conformal $\omega$-curves, subharmonicity of the norm function, and prove Theorems \ref{thm:constant-norm-means-affine} and \ref{thm:quasiregular-Stoilow}. We also finish the proof of Theorem \ref{thm:realanalytic} in this section.

After these sections, we prove the equivalence between the Liouville property and inner Möbius rigidity (\Cref{thm:inner-Mobius-Liouville:rigid}), and move to discuss the proof of the Liouville theorems (Theorems \ref{thm:codim-2} and \ref{thm:classification-dim-six}) in Section \ref{sec:Liouville-property}. The examples of conformal $\omega$-curves related to the Special Lagrangian are discussed in Section \ref{sec:examples}. We finish with a discussion on properties of the space $\mathcal{I}_n(\R^m)$ in Section \ref{sec:HPP}.

\subsection*{Acknowledgements}

We thank Ilkka Holopainen for valuable discussions on minimal surfaces and Susanna Heikkil\"a for comments on the manuscript.

%%%%%%%%%%%%%%%%%%%%%%%%%%%%%%%%%%%%%%%%%%%%%%%%%%%%%%%%%%%%%%%%%%%%%%%%%%%%%%%
%%%%%%%%%%%%%%%%%%%%%%%%%%%%%%%%%%%%%%%%%%%%%%%%%%%%%%%%%%%%%%%%%%%%%%%%%%%%%%%
%%%%%%%%%%%%%%%%%%%%%%%%%%%%%%%%%%%%%%%%%%%%%%%%%%%%%%%%%%%%%%%%%%%%%%%%%%%%%%%

\section{Preliminaries}
\label{sec:preliminaries}

\subsection{Hadamard inequality}
Let $\omega \in \Lambda^n \R^m$ be a calibration. We note first that a conformal $\omega$-curve is a $1$-quasiregular $\omega$-curve. 
For the equivalence of these definitions, we recall the following version of Hadamard's inequality for calibrations, where $\norm{\cdot}_{HS}$ is the Hilbert--Schmidt norm. That is, $\norm{A}_{HS}^2 = \mathrm{tr} A^t A$.

\begin{lemma}
\label{lemma:LA}
Let $\omega \in \Lambda^n \R^m$ be a calibration and let $A\colon \R^n \to \R^m$ be a linear map. Then
\[
|\star A^{*} \omega| \le \frac{1}{n^{n/2}} \norm{A}_{HS}^n \leq \norm{A}^n.
\]
Furthermore, $\star A^{\star}\omega = \frac{1}{n^{n/2}} \norm{A}_{HS}^n$ if and only if $A \in \CO( \omega )$ and $- \star A^{\star}\omega = \frac{1}{n^{n/2}} \norm{A}_{HS}^n$ if and only if $A = A' C$ for $A' \in \CO( \omega )$ and an orientation-reversing isometry $C \colon \R^n \to \R^n$. In either of these cases, $\frac{1}{n^{n/2}} \norm{A}_{HS}^n = \|A\|^n$.
\end{lemma}
\begin{proof}
We have that $\norm{A}_{HS}^{2} = \mathrm{tr}( A^t A ) = \sum_{ i = 1 }^{ n } |a_i|^2$, where $|a_i|$ is the Euclidean norm of the $i$'th column vector $a_i$ of $A$ . Here $| a_i |^2 \leq \|A\|^2$ for the operator norm $\|A\|$ so
\begin{equation*}
    \frac{1}{n^{n/2}} \norm{A}_{HS}^n \leq \norm{A}^n.
\end{equation*}
We next prove the inequality 
\begin{equation}\label{eq:Hadamard:HS:calib}
    |\star A^{*} \omega| \le \frac{1}{n^{n/2}} \norm{A}_{HS}^n
\end{equation}
and claim that equality holds if and only if $A \in \CO( \omega )$ or $A = A' C$ for $A' \in \CO( \omega )$ and an orientation-reversing linear isometry $C \colon \R^n \to \R^n$.

We consider the intermediate inequality
\begin{equation}\label{eq:Hadamard:HS}
    \det( B ) \leq \frac{ 1 }{ n^{n/2} }\norm{B}_{HS}^n
\end{equation}
for linear maps $B \colon \R^n \to \R^n$. We also prove that the equality \eqref{eq:Hadamard:HS} holds if and only if $B \in \CO(n)$.

By Hadamard's inequality for linear maps $B \colon \R^n \to \R^n$ and an orthonormal basis $(e_1,\ldots, e_n)$ of $\R^n$, we have that $|\det B| \le |Be_1|\cdots |Be_n|$, with equality if and only if the vectors $( B e_1, \dots, B e_n )$ are orthogonal. By the arithmetic-geometric inequality, $|Be_1|\cdots |Be_n| \le \frac{1}{n^{n/2}} \norm{B}_{HS}^n$, with equality if and only if $| B e_1 | = | B e_2 | = \dots = | B e_n |$. So inequality \eqref{eq:Hadamard:HS} holds and the equality holds if and only if $B \in \CO(n)$.

We prove \eqref{eq:Hadamard:HS:calib}. To this end, let $A = O B$ be a polar decomposition of $A$. That is, $O \colon \R^n \to \R^m$ is a linear isometry and $B \colon \R^n \to \R^n$ has a non-negative determinant. Then $\norm{A} = \norm{B}$ and $\norm{A}_{HS} = \norm{B}_{HS}$. Since $O$ is an isometry, we further have that $|\star O^*\omega| = \norm{O^*\omega}_{\comass} \le \norm{\omega}_\comass = 1$. Thus 
\[
A^*\omega = B^* O^*\omega = B^* (\lambda \vol_{\R^n}) = \lambda B^*\vol_{\R^n} = \lambda \det B \vol_{\R^n},
\]
where $-1 \le \lambda \le 1$. We conclude that
\[
|\star A^{*}\omega| = |\lambda| \det B  \leq \det B
\le \frac{1}{n^{n/2}} \norm{B}_{HS}^n
= \frac{1}{n^{n/2}} \norm{A}_{HS}^n.
\]
The claimed inequality \eqref{eq:Hadamard:HS:calib} follows. We consider now the equality conditions. If $\frac{1}{n^{n/2}}\norm{A}_{HS}^n = |\star A^{*}\omega|$, then either $A = 0$ and thus $A \in \CO( \omega )$ or $A \neq 0$ and $B$ satisfies \eqref{eq:Hadamard:HS} with equality and $|\lambda|=1$. So $B \in \CO(n) \setminus \left\{0\right\}$ in the latter case and by definition of $\lambda$, either $A \in \CO( \omega )$ (if $\lambda = 1$) or $A = A' C$ for $A' \in \CO( \omega )$ and an orientation-reversing linear isometry $C \colon \R^n \to \R^n$ (if $\lambda = -1$).

Conversely, if $A \in \CO( \omega )$, then $A = r P$ for $P \in \SO( \omega )$ and $r \geq 0$. Therefore, by definition of $\CO(\omega)$, $\star A^{*}\omega = r^n$, $\norm{A}^n = r^n$, and $\frac{1}{n^{n/2}} \norm{A}_{HS}^n = r^n$. If $A = A' C$ for $A' \in \CO( \omega )$ and an orientation-reversing linear isometry $C \colon \R^n \to \R^n$, then $\star A^{*}\omega = -\|A'\|^n$, $\norm{A}^n = \|A'\|^n$, and $\frac{1}{n^{n/2}} \norm{A}_{HS}^n = \|A'\|^n$. The claim follows.
\end{proof}

As an immediate corollary, we have the following characterization of elements of $\CO(\omega)$ which identifies conformal $\omega$-curves with $1$-quasiregular $\omega$-curves; see \cite[Lemma 4.1]{Hei:Pan:Pry:23} for a similar lemma.

\begin{corollary}
\label{cor:LA}
Let $\omega \in \Lambda^n \R^m$ be a calibration and let $A\colon \R^n \to \R^m$ be a linear map. Then the following are equivalent:
\begin{enumerate}
\item $A \in \CO(\omega)$,\label{item:LA-1}
\item $\norm{A}^n= \star A^*\omega$, and \label{item:LA-2}
\item $\frac{1}{n^{n/2}} \norm{A}_{HS}^n = \star A^*\omega$. \label{item:LA-3}
\end{enumerate}
In any of these cases, $\|A\|^n = \frac{1}{n^{n/2}} \norm{A}_{HS}^n$.
\end{corollary}

\subsection{Calibrations}

In what follows, we use the notation $\omega|_{V}$ when refering to the pullback $\iota^{*}\omega$ of $\omega$ along the inclusion map $\iota \colon V \rightarrow \R^m$. As an elementary fact on calibrations, we recall that the Hodge duality preserves comass norm and hence calibrations. 
\begin{lemma}
\label{lemma:calibration-duality}
Let $\omega \in \Lambda^n \R^m$ be a calibration. Then $\star \omega \in \Lambda^{m-n} \R^m$ is a calibration. Moreover, if $\xi$ and $\zeta$ are equivalent calibrations in $\Lambda^n \R^m$, then $\star \xi$ and $\star \zeta$ are equivalent calibrations in $\Lambda^{m-n} \R^m$. The corresponding statament is true for face equivalence.
\end{lemma}

\begin{proof}
For both claims, it suffices to observe the equality $\norm{\omega|_V}_\comass = \norm{\star \omega|_{V^\bot}}_\comass$ for an $n$-dimensional subspace $V\subset \R^m$ and its orthogonal complement $V^\bot$.  
\end{proof}

One way to generate calibrations on $\mathbb{R}^m$ is by pulling back calibrations on subspaces. The following lemma shows how their calibrated submanifolds are related. We recall that an $n$-dimensional submanifold $N \subset \R^m$ is \emph{$\omega$-calibrated}, for a calibration $\omega \in \Lambda^n \R^m$, $1 \leq n \leq m$, if $T_x N \in \Gr(\omega)$ for every $x \in N$.

\begin{lemma}\label{lemm:affinetranslation}
Suppose that $\pi \colon \R^m \to V \subset \R^m$ is an orthogonal projection to an $\ell$-dimensional subspace, $\iota \colon V \to \mathbb{R}^{\ell}$ is a linear isometry, and $\tau \in \Lambda^n \mathbb{R}^{\ell}$ is a calibration for $m \geq \ell \geq n \geq 2$. If $\omega = ( \iota \circ \pi )^{\star}\tau$ and $N \subset \R^m$ is a connected $\omega$-calibrated submanifold, then a translation of $N$ is contained in $V$. More precisely, there exists a $\tau$-calibrated submanifold $Q \subset \R^{\ell}$ and $y_0 \in \R^m$ such that $T(x) = y_0 + \iota^{-1}(x)$, $x \in \mathbb{R}^{\ell}$, satisfies $T(Q) = N$.
\end{lemma}
\begin{proof}
Let $p \in N$ and let $\pi^{\perp} \colon \mathbb{R}^m \to V^{\perp}$ denote the projection to the orthogonal complement of $V$. If $\phi$ is the composition of $\pi^{\perp}$ and the exponential map of $N$ at $p$, then the differential of $\phi$ is identically zero since the tangent spaces of $N$ are contained in $V$. So $\pi^{\perp}$ is locally constant on $N$ and, by connectivity of $N$, constant on $N$. Thus if $q$ is the unique point in $\pi^{\perp}( N )$, then the translation $N - q$ is contained in $V$ by construction. The existence of $T$ follows by unwinding the definitions.
\end{proof}
We denote $\omega \in \Lambda^n V \subset \Lambda^n \R^m$ for a subspace $V \subset \R^m$ if $\omega = \pi^{*}\tau$ for some $\tau \in \Lambda^n V$ and the orthogonal projection $\pi \colon \R^m \to V$.
\begin{corollary}\label{cor:faceequivalence}
If $V \subset \R^m$ is a subspace of dimension $\ell$ and $\omega \in \Lambda^n V \subset \Lambda^n \R^m$ is a calibration, then $\omega$ is face equivalent to $\omega' \in \Lambda^n\left( \left\{0\right\} \times \mathbb{R}^{\ell} \right)$ and an $\omega'$-calibrated submanifold is contained in $\left\{ c \right\} \times \mathbb{R}^{\ell}$ for some $c \in \R^{m-\ell}$.
\end{corollary}

Another way to generate further calibrations is as follows; see \cite[Lemma 4.2]{Hei:Pan:Pry:23} for a special case.
\begin{lemma}\label{lemm:decompositionofcalibrations}
Let $\mathbb{R}^{m} = \mathbb{R}^{ p_1 } \times \dots \times \mathbb{R}^{ p_l }$ for $p_1, \dots, p_l \geq 3$ and $l \geq 2$ and consider the coordinate projection $\pi_j \colon \R^m \to \R^{p_j}$ for $1 \leq j \leq l$. If $n \geq 3$ and $\omega_{j} \in \Lambda^n \R^{ p_j }$, then
\begin{equation*}
    \omega = \sum_{ j = 1 }^{ l } \pi_{ j }^{*}\omega_{j}
\end{equation*}
satisfies $\|\omega\|_{\comass} = \max_{ 1 \leq j \leq l } \| \omega_j \|_{\comass}$. In fact, when $\omega$ is a calibration, then
\begin{equation*}
    \Gr( \omega ) = \bigcup_{ j \in I } \Gr( \pi_j^* \omega_j),
\end{equation*}
where $I = \left\{ 1 \leq j \leq l \mid \| \omega_j \|_{\comass} = 1 \right\}$. In particular, if $A \in \CO( \omega )$, then there exists $j_0 \in I$ such that $\pi_{ j_0 } \circ A \in \CO( \omega_{ j_0 } )$ and $\pi_j \circ A = 0$ for $j \in \left\{1,\dots,l\right\} \setminus \left\{j_0\right\}$.
\end{lemma}
\begin{proof}
We first observe that $\|\omega\|_{\comass} \geq \max_{ 1 \leq j \leq l } \| \omega_j \|_{\comass}$ as is immediate from the definitions. Indeed, if $A_j \colon \R^n \to \R^{ p_j }$ is an orthogonal linear map such that $\|\omega_j\|_{\comass} = \star A_j^{*}\omega_j$ for some $1 \leq j \leq n$, then the map $A \colon \R^n \to \R^m$ for which $\pi_j \circ A = A_j$ and $\pi_{ k } \circ A = 0$ for $k \neq j$ is orthogonal and satisfies
\begin{equation*}
    \|\omega\|_{\comass}
    \geq
    \star A^{*}\omega
    =
    \star A^{*}_j\omega_j
    =
    \|\omega_j\|_{\comass}.
\end{equation*}
In case $\| \omega \|_{\comass} = 0$, the equality $\|\omega\|_{\comass} = \max_{ 1 \leq j \leq l } \| \omega_j \|_{\comass}$ is clear, so we may assume that $\| \omega \|_{\comass} > 0$ and, in fact, that $\| \omega \|_{\comass} = 1$ after normalization. After the normalization, we claim that the equality $1 = \max_{ 1 \leq j \leq l } \|\omega_j\|_{\comass}$ holds.

We consider $A \in \SO( \omega )$. Now, by \Cref{lemma:LA}, we have that
\begin{align*}
    n^{-n/2} \|A\|_{HS}^{n}
    &=
    \star A^{*}\omega
    =
    \sum_{ j = 1 }^{ l } \star A_{j}^{*}\omega_j
    \leq
    n^{-n/2} \sum_{ j = 1 }^{ l } \| \omega_j\|_{\comass} \|A_j\|_{HS}^n,
\end{align*}
where $A_j = \pi_j \circ A$ for $1 \leq j \leq l$. Hence, on the one hand, we obtain that
\begin{align*}
    \left( \sum_{ j = 1 }^{ l } \|A_j\|_{HS}^{2} \right)^{\frac{n}{2}}
    =
    \|A\|_{HS}^{n}
    \leq
    \left( \sum_{ j = 1 }^{ l } \| \omega_j\|_{\comass}\|A_j\|_{HS}^n \right).
\end{align*}
On the other hand, the elementary inequality $( \sum_{ j = 1 }^{ l } a^n_j )^{1/n} \leq ( \sum_{ j = 1 }^{ l } a^{2}_j )^{1/2}$ for $n \geq 3$ and $a_1, \dots, a_l \geq 0$, together with $\max_{ 1 \leq j \leq l }\| \omega_j \|_{\comass} \leq 1$, imply the equalities
\begin{equation*}
    \left( \sum_{ j = 1 }^{ l } \|A_j\|_{HS}^{2} \right)^{\frac{n}{2}}
    =
    \left( \sum_{ j = 1 }^{ l } \| \omega_j\|_{\comass}\|A_j\|_{HS}^n \right)
    =
    \left( \sum_{ j = 1 }^{ l } \|A_j\|_{HS}^n \right).
\end{equation*}
Considering the intersection of the boundaries of the unit balls of $\ell^2$- and $\ell^{n}$-norms gives that there exists a unique index $1 \leq j_0 \leq l$ such that $\| A \|_{HS}^{n} = \| A_{ j_0 } \|_{HS}^n$ and $\| \omega_{j_0} \|_{\comass} = 1$ (and so $A_j = 0$ for every $j \neq j_0$). 

Now that we have $\|\omega\|_{\comass} = \max_{ 1 \leq j \leq l } \| \omega_j \|_{\comass}$, the computations above imply that if $\omega$ is a calibration, then
\begin{equation*}
    \Gr( \omega ) = \bigcup_{ j \in I } \Gr( \pi_j^* \omega_j ),
\end{equation*}
where $I = \left\{ 1 \leq j \leq l \mid \| \omega_j \|_{\comass} = 1 \right\}$. The claimed property for $\CO( \omega )$ follows similarly to the $\SO(\omega)$ case above.
\end{proof}

Recall that it is a famous open problem, due to Federer \cite[1.8.4]{Fed:69}, whether the exterior product $\pi_1^*\xi \wedge \pi_2^*\zeta$ of calibrations $\xi\in \Lambda^n \R^m$ and $\zeta\in \Lambda^p \R^r$ is a calibration in $\mathbb{R}^m \times \mathbb{R}^r$. By a result of Federer (see \cite[1.8.4]{Fed:69}), this product property holds if either of the calibrations $\xi$ or $\zeta$ is simple. In particular, $\pi_1^*\omega \wedge \pi_2^*\vol_{\R^r}$ is a calibration if $\omega$ is a calibration. Moreover, by a result of Morgan \cite[Theorem 5.1]{Mor:85}, if $\min\{n,p\} \le 2$, $n\ge m-2$, $p\ge r-2$, $n=p=3$ or, $m-n = r-p = 3$, the product property holds. As far as we are aware, the general case is open.

We give a proof for \cite[1.8.4]{Fed:69} in the simple case for the reader's convenience. The claim follows by induction from the following elementary observation.
\begin{lemma}
\label{lemma:Grassmannian}
Let $\omega \in \Lambda^n\R^m$ be a calibration. Let also $\pi_1 \colon \R^m \times \R\to \R^m$ and $\pi_2 \colon \R^m \times \R \to \R$ be coordinate projections. Then
\[
\Gr(\pi_1^*\omega \wedge \pi_2^*dt) = \{ V \times \R  \colon V \in \Gr(\omega)\}.
\]
\end{lemma}
\begin{proof}
We denote $\widetilde \omega  =\pi_1^*\omega \wedge \pi_2^*dt$. Let also $(e_1,\ldots, e_m,e_{m+1})$ be  the standard basis of $\R^m \times \R$. 

Clearly, $V\times \R \in \Gr(\widetilde \omega)$ for $V\in \Gr(\omega)$. Thus it suffices to prove the converse. Let $\widetilde W\in \Gr(\widetilde \omega)$ and let $(w_1,\ldots, w_{n+1})$ be an orthonormal basis of $\widetilde W$. Then $\widetilde \omega(w_1,\ldots, w_{n+1}) =1$. By permuting the basis vectors, if necessary, we may assume that $\pi_2(w_{n+1})\ne 0$. Since $\pi_2(w_j) = \frac{|\pi_2(w_j)|}{|\pi_2(w_{n+1})|} \pi_2(w_{n+1})$ for each $j=1,\ldots, n$, we have, by multilinearity and $\widetilde \omega( w_1, \dots, w_{n}, w_{n+1} ) = 1$, that $\pi_2(w_j) = 0$ for $1\le j \le n$. Thus, we also have that $|\pi_2(w_{n+1})|=1$. We conclude that $w_j \in \R^m \times \{0\}$ for each $j=1,\ldots, n$, and that $w_{n+1}\in \{0\}\times \R$. Finally, by replacing $w_1$ by $-w_1$ and $w_{n+1}$ by $-w_{n+1}$ if necessary, we may assume that $w_{n+1} = e_{m+1}$.

Let now $W\subset \R^m$ be the subspace spanned by $\pi_1(w_1),.., \pi_1(w_n)$. Then 
\begin{align*}
\omega(\pi_1(w_1),\ldots, \pi_1(w_n)) &= \pi_1^*\omega(w_1,\ldots, w_n)\pi_2^*(dt)(e_{m+1}) 
\\
&= \widetilde \omega( w_1,\ldots, w_n, w_{n+1}) =1.
\end{align*}
Thus $W\in \Gr(\omega)$. We conclude that $\widetilde W = W\times \R$.
\end{proof}

As an application of \Cref{lemma:Grassmannian}, we prove the following technical lemma for later use. To set the stage, recall that every calibration is face equivalent to a calibration $\omega \in \Lambda^{3}\R^m$ so that $\omega( e_{m-n+1}, \dots, e_{m} ) = 1$ for the last $n$ standard coordinate vectors in the standard basis $(e_1, e_2, \dots, e_{m})$ of $\R^m$.
\begin{lemma}\label{lemm:splitting}
Let $\omega \in \Lambda^{n} \R^m$ be a calibration, for $3 \leq n \leq m$, having the property that 
\[
\omega( e_{m-n+1}, \dots, e_{m} ) = 1
\]
for the standard basis $(e_1, \dots, e_{m-n+1}, e_{m-n+2}, \dots, e_{m})$ of $\R^m$. Consider the standard coordinate projections $\pi_1 \colon \R^{m-2} \times \R^{2} \rightarrow \R^{m-2}$ and $\pi_2 \colon \R^{m-2} \times \R^{2} \rightarrow \R^{2}$. Then there exist a calibration $\alpha \in \Lambda^{n-2} \R^{m-2}$ and a form $\epsilon \in \Lambda^{n} \R^m$ satisfying $\|\epsilon\|_{\comass} \leq 2$ for which 
\begin{enumerate}
\item $\omega = \pi_{1}^{*}\alpha \wedge \pi_{2}^{*}\vol_{ \R^2 } + \epsilon$,
\item $\pi_{1}^{*}\alpha \wedge \pi_{2}^{*}\vol_{ \R^2 }( e_{m-n+1}, \dots, e_{m} ) = 1$,
\item $\epsilon( e_{m-n+1}, \dots, e_{m} ) = 0$, and 
\item for $0\le t <1$, the calibration $\omega_t = \pi_{1}^{*}\alpha \wedge \pi_{2}^{*}\vol_{ \R^2 } + t \epsilon$ is equivalent to the calibration $\pi_{1}^{*}\alpha \wedge \pi_{2}^{*}\vol_{ \R^2 }$.
\end{enumerate}
\end{lemma}
\begin{proof}
In the following proof $I \subset \left\{1,2,\dots,m\right\}$ will denote an increasing multi-index of length $n$ and
\begin{equation*}
    dx_{I} = dx_{i_1} \wedge \dots \wedge dx_{i_n}
\end{equation*}
for $I = (i_1, \dots, i_n )$. Now $\omega = \sum_{ I } \omega_I dx_I$ in the standard basis. We let $\beta' = \sum_{ m \in I } \omega_I dx_I$ and $\gamma = \sum_{ m \not\in I } \omega_I dx_I$.

We express $\gamma' = \pi^{*}\gamma$, where $\pi \colon \R^m \to \R^{m-1}$ is a coordinate projection to the first $(m-1)$ coordinates. We have 
\[
\omega = \pi^{*}\beta \wedge dx_m + \pi^{*}\gamma
\]
for some $\beta \in \Lambda^{n-1} \R^{m-1}$. By \Cref{lemma:Grassmannian}, we have that $\| \pi^{*}\beta \wedge dx_m \|_{ \comass } = \| \beta \|_{ \comass }$. Furthermore, if vectors $v_1, \dots, v_{n-1} \in \R^m$ are orthonormal, we have that
\begin{equation*}
    1 \geq \omega( v_1, \dots, v_{n}, e_m ) = \pi^{*}\beta( v_1, \dots, v_n ).
\end{equation*}
Thus $1 \geq \| \beta \wedge dx_m \|_{ \comass } = \| \beta \|_{ \comass }$ and choosing $v_j = e_{m-n+j}$, for $1 \leq j \leq n-1$, yields $\| \beta \|_{\comass} = 1$. Similarly, if vectors $w_1, \dots, w_n$ are orthonormal in $\R^{m-1}$, we have that 
\begin{equation*}
    \omega( (w_1,0), \dots, (w_n,0) )
    =
    0
    +
    \pi^{*}\gamma( (w_1,0), \dots, ( w_n,0 ) )
    =
    \gamma( w_1, \dots, w_n ).
\end{equation*}
Thus $\| \omega \|_{ \comass } = 1  \geq \| \gamma \|_{ \comass } = \| \gamma' \|_{ \comass }$, where the latter equality follows from \Cref{lemma:Grassmannian}.

We argue similarly for $\beta \in \Lambda^{n-1} \R^{m-1}$ and obtain a representation $\beta = \alpha' \wedge dx_{m-1} + \theta'$ for $\alpha' \in \Lambda^{n-2}\R^{m-1}$ and $\theta' \in \Lambda^{n-1}\R^{m-1}$, where $\alpha'$ is a calibration and the forms $\alpha'$ and $\theta'$ have the properties that $\alpha'(v_1,\ldots, v_{n-3},e_{m-1})=0$ and $\theta'(w_1,\ldots, w_{n-2},e_{m-1})=0$ for $v_1,\ldots, v_{n-3}, w_1,\ldots, w_{m-2}\in \R^{m-1}$, and that $\alpha'( e_{m-n+1}, \dots, e_{m-2} ) = 1$ and $\| \theta' \|_{\comass} \leq 1$. Then
\begin{equation*}
    \omega = \pi^{*}\alpha' \wedge dx_{m-1} \wedge dx_{m} + \pi^{*}\theta' \wedge dx_{m} + \pi^{*}\gamma.
\end{equation*}
We denote $\epsilon = \pi^{*}\theta' \wedge dx_{m} + \pi^{*}\gamma$ and write $\pi^{*}\alpha' = \pi_1^{*}\alpha$ for a calibration $\alpha \in \Lambda^{n-2} \R^{m-2}$. Since $dx_{m-1} \wedge dx_{m} = \pi_{2}^{*}\vol_{\R^n}$, we have that
\[
\omega = \pi_1^*\alpha \wedge \pi_2^*\vol_{\R^n} + \epsilon.
\]

Observe that $\omega_t( e_{m-n+1}, \dots, e_{m} ) = 1$ for every $0 \leq t \leq 1$ and $\| \omega_t \|_{ \comass } \leq 1$ by convexity of the comass norm. Thus $\| \omega_t \| = 1$ for every $0 \leq t \leq 1$. Furthermore, if $0 < t < 1$ and $\omega_t( v_1, \dots, v_n ) = 1$ for orthonormal vectors $v_1, \dots, v_n \in \R^m$, then
\begin{align*}
    1 &= \omega_t( v_1, \dots, v_n ) = t \omega_1( v_1, \dots, v_n ) + (1-t) \omega_0( v_1, \dots, v_n ) 
    \\
    &\leq t \|\omega_1\| + (1-t)\| \omega_0 \| = 1.
\end{align*}
Thus $\omega_1( v_1, \dots, v_n ) = 1 = \omega_0( v_1, \dots, v_n )$. We conclude that $\omega_{t}$ and $\omega_0$ are equivalent. The bound for the comass of $\epsilon$ follows now from the triangle inequality.
\end{proof}

%%%%%%%%%%%%%%%%%%%%%%%%%%%%%%%%%%%%%%%%%%%%%%%%%%%%%%%%%%%%%%%%%%%%%%%%%%%%%%%%
%%%%%%%%%%%%%%%%%%%%%%%%%%%%%%%%%%%%%%%%%%%%%%%%%%%%%%%%%%%%%%%%%%%%%%%%%%%%%%%%

\section{Classical regularity and the full-rank set}
\label{sec:C1}

Since conformal $\omega$-curves are $1$-quasiregular, we readily conclude that these maps are $n$-harmonic and hence $\mathcal{C}^{1,\alpha}$-regular by \cite[Theorem 5.1]{Hei:Pan:Pry:23}. Since the proof of  \cite[Theorem 5.1]{Hei:Pan:Pry:23} contains a mistake with norms in the formula corresponding to \eqref{eq:fixed}, we recall -- and fix -- the argument for the reader's convenience. We also extend the result as follows.

\begin{proposition}\label{prop:nharmonic}
Let $\omega \in \Lambda^n \R^m$ be a calibration, $\Omega \subset \R^n$ an open subset, and $F \colon \Omega \to \R^m$ a conformal $\omega$-curve. Then $F$ has an $n$-harmonic representative. In particular, $F \in \mathcal{C}^{1,\alpha}_{loc}( \Omega, \mathbb{R}^m )$. Furthermore, $F$ is real-analytic in the set $\left\{ \|DF\| > 0 \right\}$.
\end{proposition}

\begin{proof}
By Corollary \ref{cor:LA}, we have that 
\begin{equation}\label{eq:keyinequality}
    ( \star F^{*}\omega )^{ \frac{2}{n} } = \|DF\|^2 = n \|DF\|_{HS}^2
\end{equation}
almost everywhere in $\Omega$.

Next, consider an open set $U \subset \Omega$ compactly contained in $\Omega$ and let $G \in W^{1,n}_{\loc}(\Omega,\R^m)$ for which $F - G \in W^{1,n}_{0}( U, \mathbb{R}^m )$. Then $F^*\omega - G^*\omega = d\tau$, where $\tau \in W^{1,\frac{n}{n-1}}_0(\wedge^n \Omega)$; see e.g.~Iwaniec--Martin \cite{Iwa:Mar:93} or Kangasniemi \cite{Kan:21}.  Thus, by Stokes' theorem,
\[
\int_U F^*\omega = \int_U G^*\omega.
\]
Hence
\begin{equation}
\label{eq:fixed}
    \frac{ 1 }{ n^{ \frac{n}{2} } }
    \int_{ U } \|DF\|^{n}_{HS} \,d\mathcal{H}^n
    =
    \int_{ U } F^{*}\omega 
    =
    \int_{U} G^{*}\omega 
    \leq
    \frac{ 1 }{ n^{ \frac{n}{2} } }
    \int_{ U } \|DG\|^{n}_{HS} \,d\mathcal{H}^n
\end{equation}
by the pointwise Hadamard inequality for calibrations. The $n$-harmonicity of $F$ follows by the arbitrariness of $U$ and $G$. The $\mathcal{C}^{1,\alpha}$-regularity in $\Omega$ and $\mathcal{C}^{\infty}$-regularity in $\left\{ \|DF\| > 0 \right\}$ are standard for $n$-harmonic mappings, cf. \cite{Uhl:77,Har:Lin:87}.

To obtain the real-analyticity in the set $\left\{ \|DF\| > 0 \right\}$, we argue as follows. By the constant rank theorem, for every $x_0 \in \left\{ \|DF\| > 0 \right\}$, there is an open neighbourhood $U$ at which $F|_{U}$ is a $\mathcal{C}^{\infty}$-regular embedding onto its image. The image $F(U)$ is therefore a $\mathcal{C}^{\infty}$-regular $\omega$-calibrated submanifold, as its tangents are calibrated by $\omega$. Such manifolds are minimal by a result due to Harvey and Lawson \cite[Corollary 4.6]{Har:Law:82} and thus real-analytic by Reifenberg \cite{Reif:64}. Moreover, the restriction $F|_{U}$ is a conformal mapping onto the real-analytic manifold $F(U)$ and is therefore real-analytic, see e.g. \cite[Proposition A.3]{Las:Lii:Sal:22}.
\end{proof}

%We now prove \Cref{cor:uniquecont} and related extension results.

\Cref{prop:nharmonic} readily yields the followin unique continuation property for conformal $\omega$-curves.

\begin{proposition}
%[Unique extension property]
\label{cor:uniquecont}
Let $\omega \in \Lambda^{n}\mathbb{R}^m$ be a calibration, for $3\le n \le m$. Suppose that $G \colon \Omega \to \R^m$ is real-analytic with $( DG )_x \neq 0$ for every $x \in \Omega$ in a domain $\Omega \subset  \mathbb{R}^n$. Suppose that there exists an open set $\emptyset \neq U \subset \Omega$ for which $G|_{U}$ is a conformal $\omega$-curve. Then $G$ is the unique conformal $\omega$-curve extending $G|_U \colon U \to \R^m$.
%and a real-analytic immersion.
\end{proposition}

\begin{proof}
%[Proof of \Cref{cor:uniquecont}]
Consider the real-analytic mapping $G \colon \Omega \to \mathbb{R}^m$ and the real-analytic function $H(x) = ( \star G^{*}\omega )^2(x) - n^{-n} \|D_xG\|_{HS}^{2n}$ on $\Omega$ for the calibration $\omega \in \Lambda^n \mathbb{R}^m$ given by the assumption. Then $H$ vanishes on the open set $\emptyset \neq U \subset \Omega$ in which $G$ is a conformal $\omega$-curve, so $H$ is identically zero on the domain $\Omega$ by unique continuation of real-analytic functions. On the other hand, as $D_xG \neq 0$ for every $x \in \Omega$, we deduce $n^{-n}\|D_xG\|^{2n}_{HS} > 0$ for every $x \in \Omega$. So for $H$ to be identically zero, $\Omega$ splits into two components: the sets where $\star G^{*}\omega = n^{ -n/2 }\|D_xG\|^{n}_{HS} > 0$ and $\star G^{*}\omega = - n^{ -n/2 }\|D_xG\|^{n}_{HS} < 0$, respectively. By continuity of $\star G^{*}\omega$ and the connectedness of $\Omega$, only the first case happens. Thus $G$ is a conformal $\omega$-curve on $\Omega$. Lastly, we recall that the conformality of $G$ and non-degeneracy of its differential imply that $G$ is an immersion.
\end{proof}

The following proposition gives us a way to extend conformal curves that follow conformally flat submanifolds.
\begin{proposition}\label{prop:extensionresult:toamanifold}
Let $\omega \in \Lambda^n \R^m$ be a calibration, for $n \geq 3$, let $\Omega\subset \R^n$ be a domain, and let $F\colon \Omega \to \R^m$ be a nonconstant conformal $\omega$-curve whose image is contained in a connected conformally flat submanifold $N \subset \R^m$. Then there exist a domain $\hat \Omega \supset \Omega$ and a conformal map $\hat F \colon \hat \Omega \to N$ extending $F$ and satisfying $\| D\hat F \| > 0$ everywhere so that $N \setminus \hat F( \hat \Omega ) \subset N$ is relatively closed and discrete.
\end{proposition}
\begin{remark}
It can happen that $N \setminus \hat F( \hat \Omega )$ is non-empty for a maximal extension of $F$ in the setting of \Cref{prop:extensionresult:toamanifold}. Indeed, consider a Möbius transformation $G \colon \S^n \to \S^n$ for which $p = G^{-1}( \infty ) \in \S^n \setminus \left\{\infty\right\} \simeq \R^n$. Then $F \coloneqq G|_{ \mathbb{R}^n \setminus \left\{p\right\} }$ maps onto $\mathbb{R}^n \setminus \left\{ G( \infty ) \right\}$ and does not extend to an $\R^n$-valued map.
\end{remark}

\begin{proof}[Proof of Proposition \ref{prop:extensionresult:toamanifold}]
Since $F$ is not constant, there exists $x \in \Omega \cap \overline{ \left\{ \|DF\| > 0 \right\} }$. Since $N$ is conformally flat, there exists a conformal chart $(U,\varphi)$ of $N$ for which $F(x)\in U$. Let now $V\subset \Omega$ be a neighborhood of $x$ for which $F(V) \subset U$. Then $\varphi \circ F \colon U \to \R^n$ is conformal. Thus $\varphi \circ F$ is a restriction of a nonconstant M\"obius transformation $\bS^n \to \bS^n$. Since $V \subset \R^n$ and $(\varphi \circ F)(V) \subset \R^n$, we conclude that $D(\varphi \circ F) \ne 0$ in $V$. In particular, $( DF )_x\ne 0$. Consequently, $\left\{ \|DF\| > 0 \right\}$ is open and relatively closed in $\Omega$, so $\Omega = \left\{ \|DF\| > 0 \right\}$.

Next, we partially order the possible extensions $\widehat{G} \colon \widehat{\Omega} \to N$ of $F$ as follows: we require that $\widehat{G}$ is a conformal immersion defined on a domain $\widehat{\Omega}$ containing $\Omega$ and $F= \widehat{G}|_{\Omega}$. For a pair of extensions, we denote $\widehat{G}_1 \leq \widehat{G}_2$ if $\widehat{\Omega}_1 \subset \widehat{\Omega}_2$ and $\widehat{G}_1 = \widehat{G}_2|_{ \widehat{\Omega}_1 }$. An arbitrary totally ordered chain of such extensions contains a maximal element, so Zorn's lemma gives a maximal extension. We denote such a maximal extension by $\widehat{F} \colon \widehat{\Omega} \to N$.

Since $\widehat{F}$ is an open map as a local homeomorphism, the set $P \coloneqq (N\setminus \widehat{F}( \widehat{\Omega}) ) \cap \overline{ \widehat{F}( \widehat{\Omega} ) }$ is relatively closed in $N$. It remains to prove its discreteness. To this end, suppose that there exists $y_0 \in (N\setminus \widehat{F}( \widehat{\Omega}) ) \cap \overline{ \widehat{F}( \widehat{\Omega} ) }$. Let $(W,\psi)$ be a conformal chart of $N$ for which $y_0 \in W$ and denote $W' = F^{-1}( W )$.

We claim that $(N\setminus \widehat{F}( \widehat{\Omega}) ) \cap \overline{ \widehat{F}( \widehat{\Omega} ) }$ intersects $W$ only at $y_0$. This will then imply the discreteness claim. The composition $\psi \circ \widehat{F}|_{W'} \colon W' \to \psi (W)$ is a restriction of a unique nonconstant M\"obius transformation $g\colon \bS^n \to \bS^n$. Let now $W'' = g^{-1}\left( \psi(W) \right) \setminus \left\{ \infty \right\}$ and extend $\widehat{F}$ by the formula $x\mapsto (\psi^{-1} \circ g)(x)$ in $W''$. By maximality of the extension $\widehat{F}$, we conclude that $W'' \subset \widehat{\Omega}$ so by definition of $W''$, we have that $g^{-1}( \psi(y_0) ) = \infty$. The coordinate representation $\psi^{-1} \circ g|_{ W''}$ for $F$ implies that $y_0$ is isolated in $(N\setminus \widehat{F}( \widehat{\Omega}) ) \cap \overline{ \widehat{F}( \widehat{\Omega} ) }$. We also deduce that $W \setminus \left\{y_0\right\} \subset \widehat{F}( \widehat{\Omega} )$, which implies that $\widehat{F}$ maps onto $N \setminus P$ for the (relatively) closed and discrete set $P \subset N$.
\end{proof}

The extension $\hat F\colon \hat \Omega \to N$ of $F$ provided by Proposition \ref{prop:extensionresult:toamanifold} is real-analytic if the manifold $N$ is real-analytic. In this case, \emph{a posteriori}, the submanifold $N$ is $\omega$-calibrated. We record this observation as follows.

\begin{corollary}[Unique extension over submanifolds]\label{cor:uniquecont:manif}
Let $\omega \in \Lambda^{n}\mathbb{R}^m$ be a calibration,  for $3\le n \le m$, and let $N \subset \mathbb{R}^m$ be a real-analytic conformally flat submanifold. If $F \colon U \to \mathbb{R}^m$ is a nonconstant conformal $\omega$-curve mapping from a domain into $N$, there exist a domain $\Omega \supset U$ and a conformal $\omega$-curve $G \colon \Omega \to \R^m$ for which $( DG )_x \neq 0$ for every $x \in \Omega$ and $N \setminus G(\Omega) \subset N$ is relatively closed and discrete. A posteriori, $N$ is $\omega$-calibrated.
\end{corollary}

\begin{proof}
By \Cref{prop:extensionresult:toamanifold}, there exists a conformal immersion $G \colon \Omega \to N$ having non-degenerate differential and extending $F \colon U \to \R^m$ for which the set $P \coloneqq N \setminus G(\Omega) \subset N$ is relatively closed and discrete. Since $N$ is real-analytic, the conformal immersion $G$ is real-analytic by \cite[Proposition A.3]{Las:Lii:Sal:22}. Then, by \Cref{cor:uniquecont}, we have that $G$ is a conformal $\omega$-curve. Hence $N \setminus P$ is $\omega$-calibrated. Since $N \setminus P$ is dense in $N$ and the set $\Gr( \omega )$ is closed, we deduce that $N$ is $\omega$-calibrated.
\end{proof}

\subsection{Conformal curves having conical images}
\label{sec:conical-singularities}

We extend the regularity results by two results on conformal curves into conical varieties. The first result is a Reshetnyak type result on the discreteness of the fiber of the conical point, and the second a version of Varopoulos's theorem for conformal $\omega$-curves.

We say that a subset $E\subset \R^m$ has a \emph{conical singularity at $y_0\in E$} if there exists a radius $r>0$ and a closed submanifold $\Sigma \subset \bS^{m-1}(y_0,r)$ for which $E\cap B^m(y_0,r)$ is a cone over $\Sigma$, that is,
\[
E \cap B^m(y,r) = \{ ty \in B^m(x_0,r) \colon y\in \Sigma\}.
\]
We also say that $E \subset \R^m$ is a cone over a link $\Sigma \subset \bS^m(y_0,r)$ if 
\[
E= \{ y_0 + ty \in \R^m \colon y\in \Sigma\}.
\]

Since we do not use terminology related to capacities later in this article, we refer to monographs Rickman \cite[Section II.10]{Ric:93} and Heinonen--Kilpel\"ainen--Martio \cite[Section 2]{Hei:Kil:Mar:06} for the terminology and detailed discussion.

\begin{proposition}
\label{prop:conical-singularities}
Let $F \colon \Omega \to \R^m$ be a conformal $\omega$-curve for which $F\Omega$ has a conical singularity at $y_0 \in F( \Omega )$. Then $F^{-1}(0)$ is discrete in $\Omega$.
\end{proposition}
\begin{proof}
We show first that $F^{-1}(y_0)$ has $n$-capacity zero. We may assume that $y_0=0$ and that $F(\Omega) \cap B^m(0,1)$ is a cone over a submanifold $\Sigma \subset \bS^m(0,1)$, where $\Sigma$ is an orientable and smooth $(n-1)$-manifold. Let $C$ be the infinite cone over $\Sigma$, that is, $C = \{ ty \in \R^m \colon y\in \Sigma\}$ and let $u \colon \R^m \to \R$ be the function $y \mapsto -\log |y|$. Then $u$ is $n$-harmonic in the punctured cone $C\setminus \{0\}$. Indeed, the map $h \colon \Sigma \times \R \to C\setminus \{0\}$, $(z,t) \mapsto e^t z$, is a conformal homeomorphism and the composition $u \circ h \colon \Sigma \times \R \to \R$ is the $n$-harmonic function $(z,t) \mapsto -t$. 

Let now $U \subset \Omega$ be a connected component of $F^{-1}(C \cap B^m(0,1))$ 
and let $v = u \circ F \colon U \to \R$, $x\mapsto -\log|F(x)|$. 
Then $v$ is $n$-harmonic in $U\setminus F^{-1}(0)$, see e.g.~Rickman \cite[Section VI.2]{Ric:93} 
or Heinonen--Kilpeläinen--Martio \cite[Theorem 14.19]{Hei:Kil:Mar:06}, 
especially \cite[Corollary VI.2.8]{Ric:93}. Since $F$ is in $W^{1,n}_{\loc}(\Omega, \R^m)$, 
we conclude that $v \in W^{1,n}_{\loc}(U)$, see e.g.~Heinonen--Kilpel\"ainen--Martio 
\cite[Theorem 14.28]{Hei:Kil:Mar:06}. If we set the function to be $\infty$ at $F^{-1}(0)$, the extension is $n$-superharmonic and thus $F^{-1}(0) \cap U$ has $n$-capacity zero by \cite[Theorem 10.1]{Hei:Kil:Mar:06} and hence is a totally disconnected closed set in $U$. 

We are now ready to show that $F^{-1}(0)$ is discrete; let $x_0\in F^{-1}(0)$. Since $F^{-1}(0)$ is totally disconnected, it has topological dimension zero and $x_0$ has a relatively compact neighborhood $U_0$ for which $\partial U_0 \cap F^{-1}(0) = \emptyset$. Since $F(\partial U_0)$ is compact and does not meet $0$, we have that $r= \dist(0,F(\partial U_0)) >0$. It suffices to show that the set $V=F^{-1}(B^m(0,r))$ has finitely many components in $U_0$. This implies that $F^{-1}(0) \cap U_0$ is finite, thereby proving the claim.

Since $\overline{V}$ is compactly contained in $U_0$ and $F|_{\R^n\setminus F^{-1}(0)} \colon \R^n \setminus F^{-1}(0) \to C\setminus \{0\}$ is an open map, we have that $F(W)=B^m(0,r)\cap C$ for each component $W\subset V$. We orient $C\setminus \{0\}$ with the volume form $\omega|_{C\setminus \{0\}}$. Then, by change of variables, 
\[
\int_W \norm{DF}^n \ge \int_W J_F = \vol(B^n(0,r)\cap C).
\]
Since $\norm{DF}^n\in L^1(V)$, we conclude that $V$ has finitely many components. 
\end{proof}

Proposition \ref{prop:conical-singularities}, together with Picard theorems for quasiregular mappings, yield the following version of Varopoulos's theorem for conformal curves. For the statement, we recall that a finitely generated group $\Gamma$ has \emph{polynomial growth of order at most $n$} if there exists a (symmetric) generating set $S$ of $\Gamma$ for which the volume $r\mapsto \vol(B_\Gamma(e,r))$ of the balls of radius $r$ about the neutral element $e$ in the Cayley graph $\Cayley(\Gamma,S)$ is dominated by a polynomial $r\mapsto Cr^n$.

\begin{corollary}
\label{cor:Varopoulos}
Let $\omega \in \Lambda^n \R^m$ be a calibration, for $3 \leq n \leq m$, and let $N$ be a cone over link $\Sigma \subset B^m(0,1)$, where $\Sigma$ is a closed $(n-1)$-manifold, for which $N\setminus \{0\}$ is an $\omega$-calibrated submanifold. If there exists a non-constant conformal $\omega$-curve $\R^n \to N$, then the fundamental group $\pi_1(\Sigma)$ has polynomial growth of order at most $n$. 
\end{corollary}
\begin{proof}
Suppose there exists a conformal $\omega$-curve $F \colon \R^n \to N$, where $N$ is a cone over a manifold $\Sigma$ whose fundamental group $\pi_1(\Sigma)$ does not have polynomial growth of order at most $n$.
The space $N$ has one conical singularity and, by Proposition \ref{prop:conical-singularities}, $F^{-1}(0)$ is a closed set of $n$-capacity zero. Let $N_o = N\setminus \{0\}$. Then $M = \R^n \setminus F^{-1}(0)$ is $n$-parabolic in the terminology of Coulhon, Holopainen, and Saloff-Coste \cite{Cou:Hol:S-C:01} and simply connected by Rickman \cite[Lemma III.5.1]{Ric:93}. Let $H = F|_M \colon M\to N\setminus \{0\}$ and consider its lift $\widetilde H \colon M \to \widetilde{N\setminus \{0\}}$ to the universal cover $\widetilde{N_o}$ of $N_o$. Then $\widetilde{N_0}$ is $n$-hyperbolic; see the proof of \cite[Theorem 1.3]{Pan:Raj:11}. Thus, by \cite[Proposition 5.1]{Cou:Hol:S-C:01}, the mapping $\widetilde H$ is constant. This is a contradiction. 
\end{proof}

\subsection{Elementary factorization of curves having affine image}

We state now formally an elementary -- but essential -- factorization theorem for conformal curves, which is used repeatedly in the following sections. 

\begin{theorem}[Elementary factorization]
\label{thm:elementary-factorization} 
Let $\omega \in \Lambda^n \R^m$ for $2 \leq n \leq m$, let $\Omega \subset \R^n$ be a domain, and let $F\colon \Omega \to \R^m$ be a conformal $\omega$-curve whose image $F(\Omega)$ is contained in an $n$-dimensional affine subspace of $\R^m$. Then there exists an affine map $A\colon \R^n \to \R^m$, $x\mapsto y_0 + Lx$, where $y_0\in \R^m$ and $L \in \SO(\omega)$, and a conformal mapping $g \colon \Omega \to \R^n$ for which $F = A \circ g$.
\end{theorem}

\begin{proof}
By translating the image of $F$, we may assume that $0\in F(\Omega)$ and that $F(\Omega)$ is contained in an $n$-dimensional linear subspace $V$ of $\R^m$. Let $Q \in \SO(m)$ be a linear isometry for which $Q(V) = \R^n \simeq \mathbb{R}^{n} \times \left\{0\right\} \subset \R^m$ and $Q^*\vol_{\R^n} = \vol_V = \omega|_V$. Let $f=Q \circ F \colon \Omega \to \R^n$. Then $f\in W^{1,n}_{\loc}(\Omega, \R^n)$ and $(Df)_x\in \CO(n)$ for almost every $x\in \Omega$. Thus $f$ is a conformal map. Now $F = Q^{-1}|_{\R^n} \circ f$, where $Q^{-1}|_{\R^n} \in \SO(\omega)$.
\end{proof}

%%%%%%%%%%%%%%%%%%%%%%%%%%%%%%%%%%%%%%%%%%%%%%%%%%%%%%%%%%%%%%%%%%%%%%%%%%%%%%%

\subsection{Two-dimensional conformal $\omega$-curves are pseudoholomorphic}
\label{sec:two-dimensional-curves}
We denote by
\begin{equation*}
    \omega_{\symp}
    \coloneqq
    \sum_{ k = 1 }^{ \ell }
    dx_{2k-1} \wedge dx_{2k}
\end{equation*}
the \emph{standard symplectic form} on $\mathbb{R}^{ 2 \ell }$ for $\ell \in \left\{1,2,\dots\right\}$. Conformal $\omega_{\symp}$-curves are \emph{pseudoholomorphic}, see e.g.~\cite[Chapter 2]{McD:Sal:04}. We prove below that two-dimensional conformal $\omega$-curves, for any $\omega \in \Lambda^2 \mathbb{R}^m$, are pseudoholomorphic with respect to a symplectic form on a subspace. This is based on the following lemma on equivalence of calibrations in $\Lambda^2 \R^m$; see \cite[Theorem 7.16, p.~79]{Har:Law:82}.
\begin{lemma}
\label{lemma:symplectic-equivalence}
Let $\omega \in \Lambda^2 \R^m$ be a calibration. Then $\omega$ is equivalent to a symplectic form on a subspace. More precisely, there exists an orthogonal projection $\pi \colon \R^m \to V \subset \R^m$ and an isometric linear isomorphism $\iota \colon V \to \mathbb{R}^{2\ell}$ for which $\omega = ( \iota \circ \pi )^{*} \omega_{\symp}$ for the standard symplectic form $\omega_\symp \in \Lambda^2 \R^{2\ell}$ in $\R^{2\ell}$ for $1 \le \ell \le \lfloor m/2 \rfloor$.  
\end{lemma}

\begin{proof}
By the construction of a symplectic basis of $\omega$, there exists a $2k$-dimensional subspace $W$ of $\R^m$ and an orthonormal basis $(v_1,\ldots, v_{2k})$ of $W$ for which 
\[
\omega|_W = \lambda_1 v_1^* \wedge v_2^* + \cdots + \lambda_k v_{2k-1}^*\wedge v_{2k}^*
\]
and $\omega = \pi_W^* \omega|_W$, where $\pi_W \colon \R^m \to W$ is an orthogonal projection; see e.g.~Federer \cite[p.29]{Fed:69}. By permuting the basis elements if necessary, we may further assume that $\lambda_1 \ge \cdots \ge \lambda_k \ge 0$.

Let now $1\le \ell \le k$ be the maximal number for which $\lambda_\ell = \lambda_1$ and let $V$ be the subspace of $W$ spanned by vectors $v_1,\ldots, v_{2\ell}$. Then $\lambda_1 = \dots = \lambda_{\ell} = 1$ since $\omega$ is a calibration, cf. the proof of \Cref{lemm:decompositionofcalibrations}. Moreover, 
\[
\hat \omega = \pi_W^*(v_1^*\wedge v_2^* + \cdots + v_{2\ell-1}^* \wedge v_{2\ell}^*)
\]
satisfies $\CO(\hat \omega) = \CO(\omega)$.

Finally, let $\iota \colon V \to \mathbb{C}^{\ell}$ be the linear isometry satisfying $\iota(v_i) = e_i$ for $1 \leq i \leq \ell$ for the standard basis $(e_1,\ldots, e_{2\ell})$ of $\R^{2\ell}$. Then $( \iota \circ \pi )^*\omega_\symp = \pi^* \iota^* \omega_\symp = \hat \omega$. Thus $\omega$ is equivalent to $( \iota \circ \pi )^*\omega_\symp$.
\end{proof}

\begin{proposition}
\label{prop:2-dim-case}
Let $2 = n < m$ and consider a calibration $\omega \in \Lambda^2 \R^m$. Then there exists an even dimensional subspace $V\subset \R^m$ and an almost complex structure $J \colon V\to V$ having the following property. Let $F\colon \Omega \to \R^m$ be a conformal $\omega$-curve. Then $F = \tau|_V \circ H$, where $H \colon \Omega \to V$ is $J$-holomorphic and $\tau \colon \R^m \to \R^m$ is a translation. 
\end{proposition}

\begin{proof}
By translating the image of $F$, we may assume that $0 \in F(\Omega)$. By Lemma \ref{lemma:symplectic-equivalence}, there exists an orthogonal projection $\pi \colon \R^m \to V \subset \R^m$ and a linear isometry $\iota \colon V \to \mathbb{R}^{2\ell}$ for which $\omega = ( \iota \circ \pi )^{*} \omega_{\symp}$ for the standard symplectic form $\omega_\symp \in \Lambda^2 \R^{2\ell}$ in $\R^{2\ell}$ for $1 \le \ell \le \lfloor m/2 \rfloor$.

We have that $F \colon \Omega \to \R^m$ is a conformal $( \iota \circ \pi )^{\star}\omega_{\symp}$-curve and hence $G= \pi \circ F \colon \Omega \to \R^m$ is a conformal $\iota^*\omega_\symp$-curve. Similarly, $H = \iota \circ G$ is a $\omega_{\symp}$-curve. As $H$ is a conformal $\omega_{\symp}$-curve, it is $J_0$-holomorphic with respect to the standard complex structure $J_0$ on $\C^{\ell}$; see e.g.~\cite[Example 1.3]{Pan:20}. If we denote $J = \iota^{-1} \circ J_0 \circ \iota \colon V \to V$, then $G$ is $J$-holomorphic. This concludes the proof.
\end{proof}

%%%%%%%%%%%%%%%%%%%%%%%%%%%%%%%%%%%%%%%%%%%%%%%%%%%%%%%%%%%%%%%%%%%%%%%%%%%%%%%
%%%%%%%%%%%%%%%%%%%%%%%%%%%%%%%%%%%%%%%%%%%%%%%%%%%%%%%%%%%%%%%%%%%%%%%%%%%%%%%

\section{Isoperimetric inequality for conformal curves}
\label{sec:isoperimetry}

The following theorem is a conformal curve replacement for Reshetnyak's isoperimetric inequality \cite{Res:66:isoperi}; see \cite[Proposition 6.7]{Iko:23} for a related result.

%We give a direct proof for the affine rigidity of calibrated flat submanifolds using an isoperimetric inequality for conformal $\omega$-curves in spirit of Reshetnyak's isoperimetric inequality \cite{Res:66:isoperi}; see \cite[Proposition 6.7]{Iko:23} for a related result.
%
\begin{restatable}{theorem}{localIsom}
\label{thm:isoperimetric-inequality}
Let $\omega \in \Lambda^n \R^m$ be a calibration, $\Omega\subset \R^n$ an open set, and let $F \colon \Omega \to \R^m$ be a conformal $\omega$-curve. Then, for $\overline{B}(x,r) \subset \Omega$,
\begin{equation}
\label{eq:growthinequality}
    \int_{ B( x, r ) }
        \|DF\|^n
    \,d\mathcal{H}^n
    \leq
    A_{n-1}
    \left(
        \int_{ \partial B(x,r) } \|DF\|^{n-1} \,d\mathcal{H}^{n-1}
    \right)^{ \frac{n}{n-1} },
\end{equation}
where $A_{n-1}$ is Almgren's isoperimetric constant. If \eqref{eq:growthinequality} is an equality, then $F( B(x,r) )$ is an affine image of $B(x,r)$.
\end{restatable}
Recall that 
\begin{equation*}
    A_{n-1}
    =
    \frac{ \omega_n }{ ( n \omega_n )^{ \frac{n}{n-1} } }
    =
    \frac{1}{ n^{ \frac{n}{n-1} } \omega_n^{ \frac{1}{n-1} } }
\end{equation*}
is the constant in Almgren's isoperimetric inequality \cite{Alm:86} for integral $n$-currents in $\mathbb{R}^m$ and $\omega_{n}$ is the volume of the $n$-dimensional unit ball in $\R^n$. \Cref{thm:isoperimetric-inequality} implies the subharmonicity of $\|DF\|^{ \frac{n-2}{2} }$ for $n\ge 3$ stated in \Cref{thm:realanalytic} and subharmonicity of $\log \norm{DF}$ for $n=2$.

The proof of \Cref{thm:isoperimetric-inequality}
is based on the following special case of mass-energy inequality for conformal $\omega$-curves; see \cite{Iko:23}. We use the information that conformal $\omega$-curves are $\mathcal{C}^1$-regular. 
\begin{proposition}[Mass-energy inequality]
\label{prop:massenergy}
Let $\omega \in \Lambda^n \R^m$ be a calibration, $\Omega\subset \R^n$ an open set, and $F \colon \Omega \to \R^m$ a conformal $\omega$-curve. Then, for $\overline{B}(x,r) \subset \Omega$, 
\begin{equation}
\label{eq:growthinequality:isom}
\int_{ B(x,r) } \|DF\|^n \,d\mathcal{H}^n = M( F_{\sharp}[ B(x,r) ] ) 
\leq A_{n-1} \left( M( F_\sharp[ \partial B(x,r) ] ) \right)^{ \frac{n}{n-1} }.
\end{equation}
If \eqref{eq:growthinequality:isom} is an equality, then $F( B(x,r) )$ is an affine image of $B(x,r)$.
\end{proposition}
Given an oriented $k$-dimensional submanifold $M \subset \mathbb{R}^n$, we denote by $[ M ]$ the \emph{integral $k$-current} associated to $M$. Here the mass measure is as in Ambrosio and Kirchheim \cite[Definition 2.6]{AK:00:current}; see also \cite[Section 2.7]{Iko:23}) for discussion.
\begin{proof}[Proof of \Cref{prop:massenergy}]
Let $\overline{B}(x,r_0)$ be a closed ball contained in $\Omega$.
Then \cite[Lemma 6.1 and Proposition 6.7]{Iko:23} together with the Euclidean isoperimetric inequality due to Almgren \cite[Theorem 10]{Alm:86} imply that
\begin{equation*}
    \int_{B(x,r)} \|DF\|^n \,d\mathcal{H}^n
    =
    \int_{ B(x,r) } F^{*}\omega
    =
    M( F_{\sharp}[ B(x,r) ] )
\end{equation*}
for every $0 < r < r_0$ where $[ B(x,r) ]$ is the integral $n$-current induced by $[ B(x,r) ]$, oriented in the standard way, and $F_{\sharp}[ B(x,r) ]$ the pushforward under the $\mathcal{C}^{1}$-mapping $F$, cf. \cite[Definition 2.4]{AK:00:current} or \cite[4.1.7]{Fed:69}. Indeed, the first equality follows from $\|DF\|^n = \star F^{*}\omega$ (recall \Cref{prop:nharmonic}) while the second equality is immediate from \cite[Proposition 6.7]{Iko:23}. Now \cite[Lemma 6.1]{Iko:23} gives that
\begin{equation*}
    M( F_{\sharp}[ B(x,r) ] )
    =
    \int_{ B(x,r) } F^{*}\omega
    \leq
    A_{n-1}
    \left(
        M( F_\sharp[ \partial B(x,r) ] )
    \right)^{ \frac{n}{n-1} }.
\end{equation*}
Suppose that inequality \eqref{eq:growthinequality:isom} is an equality. Then, by the equality condition in Almgren's isoperimetric inequality \cite[Theorem 10]{Alm:86}, the integral current $F_\sharp[ B(x,r) ]$ is an affine disc and $\partial F_\sharp[ B(x,r) ] = F_\sharp[ \partial B(x,r) ]$ is the boundary of the said affine disc.
\end{proof}

\begin{proof}[Proof of Theorem \ref{thm:isoperimetric-inequality}]
Since the restriction $F|_{ \partial B(x,r) } \colon \partial B(x,r) \rightarrow \mathbb{R}^m$ is a well-defined $\mathcal{C}^1$-mapping, we have by the change of variables formula that
\[
M( F_\sharp[ \partial B(x,r) ] ) \leq \int_{ \partial B(x,r) } \|DF\|^{n-1} \,d\mathcal{H}^{n-1}.
\]
Thus, by \eqref{eq:growthinequality:isom}, we have that
\begin{align*}
\int_{B(x,r)} \norm{DF}^n \, d\mathcal{H}^n 
%= M(F_\sharp[B(x,r)]) 
&\leq A_{n-1} \left( M( F_\sharp[ \partial B(x,r) ] ) \right)^{ \frac{n}{n-1} } \\
&\le A_{n-1} \left( \int_{ \partial B(x,r) } \|DF\|^{n-1} \,d\mathcal{H}^{n-1} \right)^{\frac{n-1}{n}}.
\end{align*}
So \eqref{eq:growthinequality} holds. Moreover, if inequality \eqref{eq:growthinequality} is an equality, then also \eqref{eq:growthinequality:isom} is an equality and hence $F( B(x,r) )$ is an affine image of $B(x,r)$.
\end{proof}

\subsection{Subharmonicity of $\rho_F$}

The isoperimetric inequality for conformal curves yields that, for each conformal curve $F$, the function $\rho_F$, defined below, associated to the norm function $\norm{DF}$, is subharmonic.

For $2\le n \le m$, $\omega \in \Lambda^n \R^n$ a calibration, $\Omega \subset \R^n$ a domain, and $F\colon \Omega \to \R^m$ a conformal $\omega$-curve, we denote, for $n\ge 3$, 
\[
\rho_F = \norm{DF}^{\frac{n-2}{2}} \colon \Omega \to [0,\infty),
\]
and, for $n=2$, 
\[
\rho_F = \log \norm{DF} \colon \Omega \to [-\infty,\infty).
\]

\begin{lemma}
\label{lemma:subharmonicity}
Let $2\le n \le m$, $\omega \in \Lambda^n \R^n$ be a calibration, $\Omega \subset \R^n$ be a domain, and $F \colon \Omega \to \R^m$ a conformal $\omega$-curve. Then $\rho_F$ is subharmonic in the distributional sense.
\end{lemma}

%\todo{T. Keksin uudelleen tähän sen 'triviaalin argumentin', joka toimii samalla $n \geq 3$ ja $n = 2$.}

\begin{proof}
By the isoperimetric inequality for conformal curves (Theorem \ref{thm:isoperimetric-inequality}), the function $g = \norm{DF}^{n-1}$ satisfies the Beckenbach--Radó type inequality
\begin{equation}\label{eq:integralaverages}
    \left(
    \aint{ B( x, r ) }
        g^{ \frac{n}{n-1} }
    \,d\mathcal{H}^n
    \right)^{ \frac{n-1}{n} }
    \leq
    \aint{ \partial B(x,r) } g \,d\mathcal{H}^{n-1}
\end{equation}
for every $\overline{B}(x,r) \subset \Omega$; here $\aint{X} h \,d\nu$ refers to an integral average over the set $X$ of positive and finite measure with respect to the measure $\nu$. The inequality \eqref{eq:integralaverages} is stable under adding a positive constant to $g$. Indeed, if $g_\delta = g + \delta$ for $\delta > 0$, then
\begin{align*}
    \aint{ \partial B(x,r) } g_\delta \,d\mathcal{H}^{n-1}
    \geq
    \left( \aint{ B(x,r) } g^{ \frac{n}{n-1} } \,d\mathcal{H}^{n} \right)^{ \frac{n-1}{n} }
    +
    \delta
\end{align*}
Let $\nu$ be the restriction of $\mathcal H^n$ to $B(x,r)$ normalized to a probability measure. Then
\begin{align*}
    \left( \aint{ B(x,r) } g^{ \frac{n}{n-1} } \,d\mathcal{H}^{n} \right)^{ \frac{n-1}{n} }
    +
    \delta
    &=
    \| g_\delta - \delta \|_{ L^{ \frac{n}{n-1} }(\nu) }
    +
    \| \delta \|_{ L^{\frac{n}{n-1}}( \nu ) }
    \\
    &\geq
    \| g_\delta \|_{ L^{\frac{n}{n-1}}(\nu) }
    =
    \left( \aint{ B(x,r) } g^{ \frac{n}{n-1} }_\delta \,d\mathcal{H}^{n} \right)^{ \frac{n-1}{n} },
\end{align*}
where $\| \cdot \|_{ L^{p}( \nu ) }$ refers to the standard norm on the space of $L^{p}(\nu)$-integrable functions and the inequality follows from triangle inequality. Hence $g_\delta$ satisfies \eqref{eq:integralaverages}. We need the strictly positive lower bound and continuity of $g_\delta$ to apply \cite[Theorem 3.21 (i)]{Ekon:02} when $n \geq 3$ and \cite[Theorem 2.13. (ii)]{Ekon:02} when $n = 2$. The first theorem gives that $\rho_\delta \coloneqq g_\delta^{ \lambda }$ for $\lambda = (n-2)/(2(n-1))$ is subharmonic when $n \geq 3$ while the latter one gives that $\rho_\delta \coloneqq \log( g_\delta )$ is subharmonic when $n =2$. In either case, the subharmonic functions $\rho_\delta$ converge pointwise decreasingly to $\rho_F$ as $\delta \rightarrow 0^{+}$, so $\rho_F$ is subharmonic as claimed.
\end{proof}

We are now ready to finish the proof of \Cref{thm:realanalytic}.
\begin{proof}[Proof of \Cref{thm:realanalytic}]
Aside from the subharmonicity claim, the regularity of conformal $\omega$-curves follows from \Cref{prop:nharmonic}. The subharmonicity follows from \Cref{lemma:subharmonicity}.
\end{proof}

\subsection{Factorization of isoperimetrically extremal curves}
\label{sec:conditional}

We prove now the first consequence of the isoperimetric inequality \eqref{eq:growthinequality} and the elementary factorization theorem.

\constantNorm*

\begin{proof}
Since $DF$ is constant for an affine curve $F$, it suffices to prove that constant norm function $\norm{DF}$ implies an affine curve. This is clear if the constant is zero, so let $F\colon \Omega \to \R^m$ be a conformal $\omega$-curve for which $\norm{DF}$ is a non-zero constant function. Now we have that \eqref{eq:growthinequality} is an equality for each $\overline{B}(x_0,r) \subset \Omega$. Thus, by Theorem \ref{thm:isoperimetric-inequality}, $F( B(x_0,r) )$ is an affine $n$-dimensional disc. By the elementary factorization theorem (\Cref{thm:elementary-factorization}), we have that $F|_{B(x_0,r)} = A \circ g$, where $A \colon \R^n \to \R^m$ is an affine isometry and $g \colon B(x_0,r) \to \R^n$ is conformal; for $n\ge 3$, $g$ is M\"obius and, for $n=2$, $g$ is holomorphic. Since $\norm{Dg} = \norm{DF}$ and $\norm{DF}$ is constant, we conclude that $g$ is affine both for $n\ge 3$ and for $n = 2$. So $F$ is affine on $B(x_0,r)$. Now, if $F$ is affine on two intersecting balls, then $F$ is affine on the union of the balls. As $\Omega$ is connected, this implies that $F$ is affine on $\Omega$.
\end{proof}

As an immediate corollary of Theorem \ref{thm:constant-norm-means-affine}, we have the following factorization for conformal $\omega$-curves.
\begin{corollary}
\label{cor:norm-Stoilow}
Let $\omega \in \Lambda^n \R^m$ be a calibration, for $2\le n \le m$, and let $\Omega \subset \R^n$ be a domain. Consider a conformal $\omega$-curve $F\colon \Omega \to \R^m$ and a conformal embedding $g \colon \Omega \to \R^n$. Then $\norm{DF} = \norm{Dg}$ if and only if $F = A \circ g$, where $A \colon \R^n \to \R^m$, $x\mapsto y_0 + Lx$, for $y_0\in \R^m$ and $L \in \SO(\omega)$.
\end{corollary}

\begin{proof}
It suffices to show that $\norm{DF} = \norm{Dg}$ yields the factorization. Since $g \colon \Omega \to g(\Omega)$ is conformal, also its inverse $g^{-1} \colon g(\Omega) \to \Omega$ is conformal. Now the mapping $A = F \circ g^{-1} \colon g(\Omega) \to \R^m$ is a well-defined conformal $\omega$-curve satisfying 
\begin{align*}
\norm{DA} = (\norm{DF} \circ g^{-1})) \norm{Dg^{-1}} &= (\norm{Dg} \circ g^{-1})\norm{Dg^{-1}} 
    \\
    &= \frac{1}{\norm{Dg^{-1}}} \norm{Dg^{-1}} = 1.
\end{align*}
Thus $A$ is affine by \Cref{thm:constant-norm-means-affine}.
\end{proof}

%We are ready to prove \Cref{cor:comparison-principle}.
%\factorizationPrinciple*

We are now ready to prove a sharp version of the local inner M\"obius rigidity property for conformal $\omega$-curves mentioned in the introduction.

\begin{restatable}
%[Second unique extension property]
{corollary}{factorizationPrinciple}
\label{cor:comparison-principle}
Let $2\le n \le m$, let $\omega \in \Lambda^n \R^m$ be a calibration, and let $F\colon \Omega \to \R^m$ be a conformal $\omega$-curve, where $\Omega \subset \R^n$ is a domain. If $\|DF\| = \|DM\|$ on an open set $\emptyset \neq U \subset \Omega$ for a nonconstant Möbius transformation $M \colon \S^n \to \S^n$, then $F = A \circ M|_{\Omega}$ for an affine $A \colon \R^n \to \R^m$, $x \mapsto y_0 + L(x)$, $L \in \SO( \omega )$. In particular, $F$ is an inner Möbius curve and $\norm{DF} = \norm{DG}$.
\end{restatable}

\begin{proof}
%[Proof of \Cref{cor:comparison-principle}]
Let $F \colon \Omega \rightarrow \R^m$ be a conformal $\omega$-curve having the property that $\norm{DF} = \norm{DM}$ in an open set $U \subset \Omega$ for a non-constant Möbius transformation $M \colon \S^n \to \S^n$. We may assume that $U$ is a domain. 

By \Cref{cor:norm-Stoilow}, we have that $F|_{U} = A \circ M|_{U}$ for an affine isometry $A \colon \R^n \to \R^m$, $x\mapsto y_0 + Lx$, for $y_0\in \R^m$ and $L \in \SO(\omega)$. As $H = A \circ M|_{ \R^n \setminus \left\{ M^{-1}( \infty ) \right\} }$ is an extension of the conformal $\omega$-curve $F|_{U}$ and $H$ is real-analytic, with full-rank differential, \Cref{cor:uniquecont} implies that $H$ is a conformal $\omega$-curve.

Let $V$ be a connected component of $\left\{ \|DF\| > 0 \right\} \setminus \left\{ M^{-1}(\infty) \right\}$ containing $U$. As $\star H^{*}\omega$ and $\star F^{*}\omega$ are real-analytic on $V$ (\Cref{thm:realanalytic}) and coincide on $U \subset V$, they coincide on $V$ by unique continuation. By continuity, they also coincide on the relative closure $\overline{V} \cap \Omega \setminus \left\{ M^{-1}(\infty) \right\}$. As the differential of $H$ has full rank on $\Omega \setminus \left\{ M^{-1}(\infty) \right\}$, we have that $\overline{V} \cap \Omega \subset \left\{ \|DF\| > 0 \right\} \setminus \left\{ M^{-1}(\infty) \right\}$. So $V$ is open and relatively closed on $\Omega \setminus M^{-1}( \infty )$ and, by connectivity, $V = \Omega \setminus M^{-1}( \infty )$. We conclude that $F$ and $H$ are real-analytic on $\Omega \setminus \left\{ M^{-1}(\infty) \right\}$ that coincide on $U$ and, by unique continuation, they coincide on $\Omega \setminus \left\{ M^{-1}(\infty) \right\}$. Thus the equality $F = H$ on $\Omega \setminus \left\{ M^{-1}(\infty) \right\}$ implies that $M^{-1}( \infty ) \in \S^n \setminus \Omega$, since $\|DF\|$ is continuous on $\Omega$. The claim follows.
\end{proof}

%\Cref{cor:Maximum-principle} can be deduced directly from \Cref{lemma:subharmonicity} and \Cref{thm:constant-norm-means-affine}.

%\maximumPrinciple*

Before proving Theorem \ref{thm:quasiregular-Stoilow}, we prove the maximum principle for conformal $\omega$-curves mentioned in the introduction.

\begin{restatable}
%[Maximum principle]
{corollary}{maximumPrinciple}
\label{cor:Maximum-principle}    
Let $\omega \in \Lambda^n \R^m$ be a calibration, for $2 \leq n \leq m$, and let $F\colon \Omega \to \R^m$ be a conformal $\omega$-curve, where $\Omega \subset \R^n$ is a domain. If the norm function $\norm{DF}$ attains a positive local maximum in $\Omega$, then $\norm{DF}$ is constant on $\Omega$. In particular, $F$ is affine.
\end{restatable}

\begin{proof}
We are given a conformal $\omega$-curve $F \colon \Omega \to \R^m$ such that $\|DF\|$ has a local maximum at $x_0 \in \Omega$. Using the notation $\rho_F$ from \Cref{lemma:subharmonicity}, we deduce that the subharmonic function $\rho_F$ reaches a local maximum at $x_0 \in \Omega$. That is, there exists a domain $U \subset \Omega$ such that $U \subset \left\{ \rho_F = \rho_F(x_0) \right\}$. \Cref{thm:constant-norm-means-affine} yields that $F$ coincides with an affine inner Möbius curve on $U$ after which \Cref{cor:comparison-principle} implies that $F$ coincides with the affine inner Möbius curve on $\Omega$. The claim follows.
\end{proof}

We continue now with the proof of Theorem \ref{thm:quasiregular-Stoilow}.

\qrStoilow*

\begin{proof}
Suppose $F = A \circ f$, where $f\colon \Omega \to \R^n$ is quasiregular and $A\colon \R^n \to \R^m$, $x\mapsto y_0 + Lx$, where $L\in \SO(\omega)$, is affine. Then 
\[
(DF)^t DF = (Df)^t L^t L Df = (Df)^t Df.
\]
Suppose now that $(DF)^t DF = (Df)^t Df$ for a quasiregular map $f \colon \Omega \to \R^n$. Let $B_f = \{ x\in \Omega \colon f \text{ is not a local homeomorphism at } x\}$ be the branch set of $f$ and $\Omega' = \Omega \setminus B_f$. Since $B_f$ is a closed set of topological dimension at most $n-2$ (see e.g.~Rickman \cite[Section I.4]{Ric:93}), we have that $\Omega'$ is open, connected, and dense in $\Omega$; by definition $f$ is a local homeomorphism in $\Omega'$.

We show first that each point $x\in \Omega'$ has a neighborhood $U_x$ for which $F( U_x )$ is contained in an affine $n$-dimensional subspace.

Let $x \in \Omega'$ and let $U_x \subset \Omega'$ be a neighborhood of $x$ for which $h = f|_{U_x} \colon U_x \to f(U_x)$ is a homeomorphism. Then $h \colon U_x \to f(U_x)$ and the inverse $h^{-1} \colon f(U_x) \to U_x$ are quasiconformal. Let $A = F \circ h^{-1}$. We recall that $A \in W^{1,n}_{\loc}( f(U_x), \R^m )$ by \cite[Lemma VI.6.6]{Ric:93} and that the weak differential satisfies the chain rule
\begin{equation*}
    DA = ( (DF) \circ h^{-1} ) D(h^{-1})
    \quad\text{almost everywhere on $f(U_x)$.}
\end{equation*}
We also recall that $h^{-1}$ satisfies the Lusin's properties $(N)$ and $(N^{-1})$ (see e.g.~Rickman \cite[Theorem II.7.4(3)]{Ric:93}). So we have
\begin{equation*}
    I = ( (Dh) \circ h^{-1} ) \circ D( h^{-1} ),
    \quad
    (DA)\R^n = (DF)\R^n \in \Gr( \omega ),
\end{equation*}
and $\star A^{\star}\omega = ( \star F^{\star}\omega ) \circ h^{-1} J_{h^{-1}}\geq 0$ almost everywhere on $f(U_x)$. We also have that 
\begin{align*}
(DA)^t DA 
&= \left( ((DF)\circ h^{-1}) Dh^{-1} \right)^t 
\left((DF)\circ h^{-1}) Dh^{-1}\right) \\
&= (Dh^{-1})^t \left( (DF)^t (DF))\circ h^{-1} \right) Dh^{-1} \\
&= (Dh^{-1})^t ( (Df)^t \circ h^{-1}) ((Df) \circ h^{-1}) Dh^{-1} \\ 
&= (D(f\circ h^{-1}))^t D(f\circ h^{-1}) 
= I.
\end{align*}
In particular, $DA$ is conformal, $(DA)\R^n \in \mathrm{Gr}( \omega )$, and $\star A^{\star}\omega \geq 0$ hold almost everywhere on $f(U_x)$. So $A$ is a conformal $\omega$-curve satisfying $\|DA\| \equiv 1$. Thus, by \Cref{thm:constant-norm-means-affine}, $A$ is affine. Since $F|_{U_x} = A \circ h^{-1}$, we have that $F(U_x)$ is contained in an $n$-dimensional affine subspace.

Since each point of $\Omega'$ has a neighborhood in $\Omega'$ which is contained in an $n$-dimensional affine subspace and $\Omega'$ is connected, we conclude that $F(\Omega')$ is contained in an $n$-dimensional affine subspace $V$. Since $\Omega'$ is dense in $\Omega$ and $F$ is continuous, we have that $F(\Omega)$ is contained in $V$. The claim follows now from a quasiregular version of the elementary factorization theorem.
\end{proof}

\begin{remark}
\label{rmk:Stoilow}
Corollary \ref{cor:norm-Stoilow} may also be viewed as a special case of Theorem \ref{thm:quasiregular-Stoilow}. Indeed, for a conformal $\omega$-curve $F\colon \Omega \to \R^m$ and a conformal map $g \colon \Omega \to \R^n$, we have that $(DF)_x \in \CO(\omega)$ and hence $(DF)_x\R^n \in \Gr(\omega)$ and $\star F^{*}\omega \geq 0$ for each $x\in \Omega$. Furthermore, $(DF)^t DF = (\star F^*\omega)^{2/n} I = \norm{DF}^2 I$ and $(Dg)^t Dg = J_g^{2/n} I = \norm{Dg}^2 I$. Thus $(DF)^t DF = (Dg)^t Dg$ if $\norm{DF} = \norm{Dg}$.
\end{remark}

We finish this section with the proof of \Cref{cor:non-solvability}.

\nonSolvability*

\begin{proof}
%[Proof of \Cref{cor:non-solvability}]
Let $Z \colon \R^n \to \R^n\setminus \{0\}$ be a Zorich map; see e.g.~\cite[Section I.3.3]{Ric:93}. The map $Z$ is a quasiregular mapping, so, in particular, a branched cover, and has a full-rank differential outside a set of Hausdorff codimension two.

Consider the conformal $\omega$-curve $F \colon \R^n \setminus \{0\} \to \R^{2n}$ from \Cref{prop:SL-catenoid}. Then $H \coloneqq F\circ Z\colon \R^n \to \R^{2n}$ is a quasiregular $\omega_\SL$-curve and $H( \R^n )$ is the Lagrangian catenoid $\mathcal{M}_{\SL}$ by the surjectivity of $Z$.

Suppose that there exists a quasiregular map $f\colon \Omega \to \R^n$ for which the equality $(Df)^t Df = (DH)^t DH$ holds almost everywhere on a domain $\Omega \subset \R^n$. Since $(DH)_x \R^n\in \Gr(\omega_\SL)$ and $\star H^{\star}\omega(x) \geq 0$ for almost every $x\in \Omega$, we have, by \Cref{thm:quasiregular-Stoilow}, that $H( \Omega )$ is an affine submanifold on $\R^{2n}$. Thus $\mathcal{M}_{\SL}$ is flat. Indeed, if $U$ is a connected component of $F^{-1}( H(\Omega) )$, then $F|_{U}$ has affine image and hence is a conformal $\eta$-curve for a simple calibration $\eta \in \Lambda^n \R^m$. Since $F$ is a real-analytic mapping with full rank differential, \Cref{cor:uniquecont} implies that $F$ is a conformal $\eta$-curve and thus has an affine image. This, however, contradicts the definition of $\mathcal{M}_{\SL}$. We conclude that no such $f$ exists.
\end{proof}

%%%%%%%%%%%%%%%%%%%%%%%%%%%%%%%%%%%%%%%%%%%%%%%%%%%%%%%%%%%%%%%%%%%%%%%%%%%%%%%%
%%%%%%%%%%%%%%%%%%%%%%%%%%%%%%%%%%%%%%%%%%%%%%%%%%%%%%%%%%%%%%%%%%%%%%%%%%%%%%%%

%%%%%%%%%%%%%%%%%%%%%%%%%%%%%%%%%%%%%%%%%%%%%%%%%%%%%%%%%%%%%%%%%%%%%%%%%%%%%%%
%%%%%%%%%%%%%%%%%%%%%%%%%%%%%%%%%%%%%%%%%%%%%%%%%%%%%%%%%%%%%%%%%%%%%%%%%%%%%%%

\section{Inner M\"obius rigid calibrations}
\label{sec:Liouville-property}

In the following two subsections, we prove the equivalence of the Liouville property with inner Möbius rigidity and with \emph{conformal rigidity}, respectively. After these equivalences, we show that stabilizations of calibrations are conformally rigid and that, in general, codimension two calibrations are conformally rigid.

To define conformal rigidity, recall that a Riemannian $n$-manifold is \emph{conformally flat} if it has an atlas of charts that are conformal and \emph{flat} if it has an atlas of charts that are Riemannian isometries. 

\begin{definition}\label{def:conf:rigid}
A calibration $\omega \in \Lambda^n\mathbb{R}^m$, for $2 \leq n \leq m$, is \emph{conformally rigid} if every conformally flat $\omega$-calibrated submanifold is a flat submanifold of $\R^m$.
\end{definition}
Since every two-dimensional manifold is conformally flat due to the existence of isothermal coordinates, \Cref{def:conf:rigid} is mainly of interest in the case $n \geq 3$.

\subsection{Equivalence of the Liouville property and inner Möbius rigidity}
We are now ready to prove that a calibration has the Liouville property if and only if it is inner M\"obius rigid, that is, Theorem \ref{thm:inner-Mobius-Liouville:rigid}.

\innermob*

\begin{proof}
%[Proof of \Cref{thm:inner-Mobius-Liouville:rigid}]
We recall that if a conformal $\omega$-curve $F \colon \Omega \subset \R^n \to \R^m$ is inner Möbius curve, we have $F = A \circ g$ for affine $A(x) = x_0 + Lx$ for $L \in \CO( \omega )$ and $g(z) = M(x)$ for $x \in \Omega$ and a Möbius transformation $M \colon \S^n \to \S^n$. Notice that there is $K \in \CO(m)$ satisfying $K( x, 0 ) = L( x )$ for $x \in \R^n$. As $M$ is defined by \eqref{eq:Mobius}, we find, by identifying $\R^n$ with $\R^n \times \left\{0\right\} \subset \R^m$ as above (and $\S^n$ with $\S^n \times \left\{0\right\} \subset \S^m$), a Möbius transformation $\widehat{M} \colon \S^m \to \S^m$ extending $M$, that is, $\widehat{M}|_{ \S^n } = M$.

Suppose now that $F \colon \Omega \to \R^m$ is a nonconstant conformal $\omega$-curve that is the restriction of a Möbius transformation $\widehat{M} \colon \S^m \to \S^m$, with the identifications $\R^n \subset \R^m \subset \S^m$ as above. We show that $F$ is an inner Möbius curve. 

Let now $\widehat{\Omega} = \R^n \setminus \{ \widehat{M}^{-1}(\infty) \}$. Then $G = \widehat{M}|_{ \widehat{\Omega} } \colon \widehat{\Omega} \to \R^m$ is a real-analytic map having a nondegenerate differential and extending $F$, so \Cref{cor:uniquecont} gives that $G$ is a conformal $\omega$-curve. Since $\widehat{M}$ is conformal, we have that $\| DG \|(x) = \| D\widehat{M} \|( x, 0 )$ for $x \in \R^n$. There are two cases to consider. Either $\|DG\|$ is constant and thus $G$ is affine by \Cref{thm:constant-norm-means-affine}, or $\|DG\|(x) = \lambda^2 | (x,0) - x_0 |^{-2} = \lambda^2 ( |x-y_0|^2 + |z_0|^2 )^{-1}$ for $\lambda \in (0,\infty)$ and $x_0 = ( y_0, z_0 ) \in \R^n \times \R^{m-n}$. In the first case, $F$ being inner Möbius is immediate. For the second case, we recall from \Cref{lemma:subharmonicity} that $\|DG\|^{ \frac{n-2}{2} } \colon \widehat{\Omega} \to (0,\infty)$ is a subharmonic function. A direct computation shows that
\begin{align*}
    \Delta( \|DG\|^{ \frac{n-2}{2} } )(x)
    &=
    \frac{ \lambda^{ n - 2 }( n - 2 ) n }{ \left( |x-y_0|^2 + |z_0|^2 \right)^{ \frac{n}{2} }  }
    \left(
        \frac{ |x-y_0|^2 }{ |x-y_0|^2 + |z_0|^2 } - 1
    \right)
\end{align*}
is negative for every $x \in \widehat{\Omega}$ unless $z_0 = 0$. Thus, by subharmonicity, $z_0 = 0$.

Let now $M \colon \R^n\setminus \{y_0\} \to \R^n$, $y \mapsto R( (y-y_0)/|y-y_0|^2 )$, where $R \colon \R^n \to \R^n$, $(x_1,\ldots, x_n) \mapsto (x_1,\ldots, -x_n)$. Then $G \circ M^{-1}$ is a conformal $\omega$-curve on $\R^n$ for which $\|D( G \circ M^{-1} ) \|$ is constant. Therefore $G \circ M^{-1}$ is affine by \Cref{thm:constant-norm-means-affine}. Thus $G$, and also $F = G|_{ \Omega }$, is an inner Möbius curve. 
\end{proof}

\subsection{Equivalence of conformal and inner M\"obius rigidity}
\label{sec:conformal-rigidity}

The following theorem connects two notions of rigidity.
\begin{restatable}{theorem}{Liouville}
\label{thm:Liouville}
Let $3 \le n \le m$ and let $\omega \in \Lambda^n \R^m$ be a calibration. Then $\omega$ is conformally rigid if and only if $\omega$ is inner M\"obius rigid.
\end{restatable}
Since calibrated submanifolds are minimal \cite{Har:Law:82}, flat calibrated connected submanifolds are affine. Although this is well known, see e.g.~\cite[Proposition 3.10]{Daj:Toj:19}, we give a simple proof based on isoperimetric inequality for conformal curves for the reader's convenience.
\begin{lemma}
\label{lemma:flat-is-affine}
Let $\omega \in \Lambda^n \R^m$ be a calibration and let $M$ be a flat and connected $\omega$-calibrated submanifold. Then $M$ is contained in an $n$-dimensional affine subspace.
\end{lemma}
\begin{proof}
Let $p \in M$ and consider a Riemannian isometry $\phi_p \colon B_{\epsilon} \rightarrow B_{M}( p, \epsilon )$, where $B_\epsilon \subset \R^n$, and the inclusion map $\iota \colon M \rightarrow \mathbb{R}^m$. Then $F = \iota \circ \phi$ is a conformal $\omega$-curve. Since $\norm{DF}=1$ everywhere, the mapping $F$ is affine by Theorem \ref{thm:constant-norm-means-affine}. If two balls $B_M( p_1, \epsilon_1 )$ and $B_M( p_2, \epsilon_2 )$ overlap, the corresponding affine subspaces coincide and hence $M$ is contained in an $n$-dimensional affine subspace by connectivity of $M$.
\end{proof}

Similar argument shows that inner M\"obius rigid calibrations are conformally rigid.
\begin{lemma}\label{lemma:mob:to:conf:rigid}
Let $\omega \in \Lambda^n \mathbb{R}^m$ be an inner Möbius rigid calibration. Then $\omega$ is conformally rigid.
\end{lemma}
\begin{proof}
Let $M$ be a conformally flat $\omega$-calibrated submanifold. It suffices to prove that $M$ is flat. Let $p \in M$ and consider a conformal chart $\phi \colon U \to W \subset \mathbb{R}^n$, where $p \in U \subset M$ is a (relatively open) domain. Let $\iota \colon M \xhookrightarrow{} \mathbb{R}^m$ be the inclusion. Then $F = \iota \circ \phi^{-1}$ is a conformal $\omega$-curve since $M$ is $\omega$-calibrated and $\phi$ is conformal. In particular, the image of $F$ is contained in an $n$-dimensional affine subspace of $\mathbb{R}^m$ and thus $U$ is flat. As $M$ can be covered by such $U$, $M$ is flat and the claim follows.
\end{proof}

\begin{proof}[Proof of \Cref{thm:Liouville}]
\Cref{lemma:mob:to:conf:rigid} proved one direction of the claim. So it suffices to prove that conformal rigidity implies inner Möbius rigidity. To this end, let $F \colon \Omega \rightarrow \mathbb{R}^m$ be a conformal $\omega$-curve for a conformally rigid calibration $\omega \in \Lambda^n \R^m$ and a domain $\Omega \subset \R^n$.

We first recall from \Cref{prop:nharmonic} that $F \in \mathcal{C}^{1}( \Omega, \mathbb{R}^m )$. The claim is clear if $F$ is constant, so we may assume that there exists $x \in \left\{ \|DF\| > 0 \right\}$. Then, by \Cref{prop:nharmonic}, there exists a domain $x \in U \subset \left\{ \|DF\| > 0 \right\}$ for which the restriction $F|_{ U } \colon U \to \R^m$ is a real-analytic embedding. Since $F$ is a conformal $\omega$-curve, the image $M = F( U )$ is an $\omega$-calibrated submanifold. By \Cref{lemma:flat-is-affine}, $M$ is contained in an $n$-dimensional affine subspace $W$ and, by the elementary factorization theorem (\Cref{thm:elementary-factorization}), $F|_{U}$ is inner M\"obius. \Cref{cor:comparison-principle} now implies that $F$ is inner Möbius.
\end{proof}

%%%%%%%%%%%%%%%%%%%%%%%%%%%%%%%%%%%%%%%%%%%%%%%%%%%%%%%%%%%%%%%%%%%%%%%%%%%%%%%
%%%%%%%%%%%%%%%%%%%%%%%%%%%%%%%%%%%%%%%%%%%%%%%%%%%%%%%%%%%%%%%%%%%%%%%%%%%%%%%

\subsection{Conformal rigidity of stabilizations of calibrations}
\label{sec:products}

In this section, we prove the following theorem.

\begin{theorem}
\label{thm:product}
Let $\omega \in \Lambda^n\R^m$ be a calibration for $2\le n \le m$ and $r\ge 1$. Let also $\pi_1 \colon \R^m \times \R^r \to \R^m$ and $\pi_2 \colon \R^m \times \R^r \to \R^r$ be coordinate projections. Suppose that one of the conditions 
\begin{enumerate}
\item $r \ge 2$,
\item $\omega$ is conformally rigid,
\item $n=2$, or 
\item $\omega = \omega_\SL$
\end{enumerate}
is satisfied. Then the calibration $\widetilde \omega = \pi_1^*\omega\wedge \pi_2^*\vol_{\R^r} \in \Lambda^{n+r}\R^{m+r}$ is conformally rigid.
\end{theorem}

Before proving \Cref{thm:product}, we state a corollary, which yields the density of conformally rigid calibrations among all calibrations. Recall that, by \Cref{thm:Liouville}, conformal rigidity is equivalent to inner Möbius rigidity. 

\begin{corollary}\label{cor:applicability}
Let $\omega \in \Lambda^n \R^m$ be a calibration for $3 \leq n \leq m$. Then there exists a conformally rigid calibration $\widetilde \omega\in \Lambda^n \R^m$ and a form $\epsilon\in \Lambda^n \R^m$ having comass at most $2$ for which $\omega = \widetilde \omega + \epsilon$ and $\widetilde \omega + t\epsilon$ is equivalent to $\widetilde \omega$ for $0\le t < 1$.
\end{corollary}

\begin{proof}
Let $O \in \SO(m)$ be an isometry for which the calibration $O^*\omega \in \Lambda^n \R^m$ satisfies the assumptions of \Cref{lemm:splitting}, that is, $(O^*\omega)(e_{m-n+1},\ldots, e_m) =1$ for the standard basis $(e_1,\ldots, e_m)$ of $\R^m$. Then, by \Cref{lemm:splitting}, $O^{*}\omega = \widetilde{\omega}' + \epsilon'$, where $\widetilde{\omega}'$ is a calibration, $\|\epsilon'\|_{\comass} \leq 2$, and $\widetilde{\omega}' + t \epsilon'$ is a calibration equivalent to $\widetilde{\omega}'$ for $0 \leq t < 1$. Furthermore, $\widetilde{\omega}'$ is simple when $n = 3$ or satisfies assumption (1) in \Cref{thm:product} for $n \geq 4$. In either case, $\widetilde{\omega}'$ is conformally rigid. We may now take $\widetilde{\omega} = ( O^{-1} )^{*}\widetilde{\omega}'$ and $\epsilon = ( O^{-1} )^{*}\epsilon'$. The claim follows.
\end{proof}

In the following proposition, we prove basic properties for calibrated submanifolds of $\widetilde \omega = \pi_1^*\omega\wedge \pi_2^*\vol_{\R^r}$. Recall that, by \Cref{lemma:Grassmannian}, we have $\Gr(\widetilde \omega) = \{ V \times \R^r \colon V\in \Gr(\omega)\}$.

\begin{proposition}
\label{prop:locflat-new}
Let $\omega \in \Lambda^n\R^m$ be a calibration for $2\le n \le m$, $r\ge 1$, and let $\pi_1 \colon \R^m \times \R^r\to \R^m$ and $\pi_2 \colon \R^m \times \R^r \to \R^r$ be coordinate projections. Let also $\widetilde \omega = \pi_1^*\omega\wedge \pi_2^*\vol_{\R^r}\in \Lambda^{n+r} \R^{m+r}$, let $M \subset \R^{m+r}$ be an $\widetilde \omega$-calibrated submanifold, and consider the distribution
\[
\Delta = M \times (\{0\}\times \R^r) \subset TM.
\]
Then the following properties hold:
\begin{enumerate}
\item The distribution $\Delta$ is fixed by the holonomy group of $M$. Furthermore, $\Delta$ and its orthogonal complement $\Delta^\perp \subset TM$ are parallel distributions and their leaves are totally geodesic in $M$. \label{item:localflat-now-1} 
\item The orthogonal projection $\pi_2$ restricts to a Riemannian local isometry on the leaves $L_p$, for $p \in M$, of $\Delta$. Moreover, $L_p = \left\{ \pi_1(p) \right\} \times \pi_2( L_p )$, for $p\in M$.
\label{item:localflat-now-2}
\item The leaves of $\Delta^{\perp}$ are in one-to-one correspondence with the connected components of the fibers $M \cap \pi_2^{-1}(p)$ for $p \in M$. That is, there exists a unique connected $\omega$-calibrated submanifold $N_p \subset \R^m$ for which $N_p \times \left\{ \pi_2(p) \right\}$ is the maximal leaf of $\Delta^{\perp}$ containing $p$. \label{item:localflat-now-3}
\end{enumerate}
\end{proposition}

We conclude that an $\widetilde \omega$-calibrated submanifold is locally a Riemannian product; see e.g.~Kobayashi and Nomizu \cite[Chapter IV, Proposition 5.2.]{Kob:Nom:96}.

\begin{corollary}
\label{cor:locflat-new}
Let $\omega \in \Lambda^n\R^m$ be a calibration for $2\le n \le m$, $r\ge 1$, and let $\pi_1 \colon \R^m \times \R^r\to \R^m$ and $\pi_2 \colon \R^m \times \R^r \to \R^r$ be coordinate projections. Let also $\widetilde \omega = \pi_1^*\omega \wedge \pi_2^*\vol_{\R^r}\in \Lambda^{n+r} \R^{m+r}$, and let $M \subset \R^{m+r}$ be an $\widetilde \omega$-calibrated submanifold. 

Then, for each $p \in M$, there exist a neighborhood $V$ of $p$ in $M$, an $\omega$-calibrated submanifold $N \subset \R^m$, an open set $U \subset \R^r$, and a Riemannian isometry $\phi \colon N \times U \to M$ onto $V$.
\end{corollary}

\begin{proof}[Proof of Proposition \ref{prop:locflat-new}]
Let $(e_1,\ldots, e_{m+r})$ be the standard basis of $\R^{m+r}$ and, for $j=1,\ldots, r$, let $X_j \colon M \to \R^{m+r}$ be a restriction of the coordinate vector field $p \mapsto e_{m+j}$.

We recall that the Levi-Civita connection $\nabla$ on $M$ coincides with the tangential projection of the Levi-Civita connection on $\mathbb{R}^m$, see e.g. \cite[Chapter 5]{Lee:18}. Thus, if we consider any piecewise smooth curve $\gamma \colon [a,b] \to M$, we have
\begin{equation}\label{eq:levicivita}
    \nabla_{ \dot{\gamma}(t) } X_j = 0
    \quad\text{for every differentiability point of $\gamma$ and $1 \leq j \leq r$.}
\end{equation}
Indeed, \eqref{eq:levicivita} holds for the Euclidean Levi-Civita connection and therefore also for its tangential projection $\nabla$ since $X_j$ takes values on $TM$. Thus the parallel transport along smooth curves acts as the identity on each $X_j$, see e.g. \cite[Corollary 4.33, Theorem 4.34, and Corollary 4.35]{Lee:18}, for every $1 \leq j \leq r$. In particular, the holonomy group at a point $p \in M$ acts on $\Delta_p$ as the identity. This implies that $\Delta$ is a parallel distribution and the metric compatibility of the Levi--Civita connection gives that $\Delta^{\perp}$ is parallel as well. Recalling that the Levi--Civita connection is torsion-free, the parallel distributions are integrable; see e.g. \cite[Proposition 5.5]{Lee:18} for the quoted properties of the connection. Furthermore, the leaves are totally geodesic in $M$ by \cite[Proposition 5.1 (2),  p. 180]{Kob:Nom:96}. Thus \eqref{item:localflat-now-1} follows.

Let $L_p$ and $L_{p}^{\perp}$ be the maximal (connected) integral manifolds of $\Delta$ and $\Delta^{\perp}$, respectively, at $p \in M$. The restriction of the orthogonal projection $\pi_2$ to $L_p$ is a local Riemannian isometry, as the differential of the restriction is an isometry. Combining this with $L_p$ being totally geodesic in $M$, we deduce that the exponential map at each $q \in L_p$ acts as the affine map $t \mapsto q + tv$ for every $v \in \Delta_q = T_q L_p$. Hence, for every $p \in M$, the open set $U_p = \pi_2( U_p ) - \pi_2( p ) \subset \mathbb{R}^{ r }$ satisfies
\begin{equation*}
    L_p = \left\{ p + (0,v) \mid v \in U_p \right\}.
\end{equation*}
We conclude that \eqref{item:localflat-now-2} holds.

To finish, we observe that every $p \in M$ is contained in $L_p^{\perp}$; so, in particular, in some leaf of $\Delta^{\perp}$. Moreover, as $L_p^{\perp}$ is totally geodesic and connected, given any $q \in L_p^{\perp}$, there exists a smooth curve $\gamma \colon [0,1] \rightarrow L_p^{\perp}$ joining $p$ to $q$. Since the differential of $\pi_2$ is identically zero on the tangents of $L_p^{\perp}$, we deduce that $\pi_2 \circ \gamma$ is constant. This implies that $L_p^{\perp} \subset \pi_2^{-1}( \pi_2(p) )$. Conversely, if $K$ is a connected component of $M \cap \pi_2^{-1}( \pi_2(p) )$, then 
\begin{equation*}
    K = \bigcup_{ p \in K } L_p^{\perp}.
\end{equation*}
The invariance of domain yields that such $L_p^{\perp}$ are (relatively) open in $K$ so the connectivity of $K$ implies that $K = L_p^{\perp}$ for a point $p \in K$. Now \eqref{item:localflat-now-3} is immediate by considering the definition of $\Delta^{\perp}$ and the assumption we have on the tangent bundle $TM$.
\end{proof}

\Cref{thm:product} now follows from a result of Yau on conformally flat Riemannian manifolds and a related result of Ejiri on totally real submanifolds.

\begin{proposition}[{\cite[Theorem 4]{Yau:73}}]
\label{prop:conformallyflat:product}
Let $M = N_1 \times N_2$ be a conformally flat Riemannian manifold of dimension $n \geq 3$. Then, either one of the factors $N_1$ and $N_2$ is one-dimensional and the other factor has constant sectional curvature, or the dimensions of the factors are at least two and the factors $N_1$ and $N_2$ have constant sectional curvature $\kappa_1$ and $\kappa_2$, respectively, with $\kappa_1 = -\kappa_2$.
\end{proposition}

\begin{proposition}[{\cite{Eji:82}}]
\label{prop:conformallyflat:product-real}
Let $N$ be a totally real, minimal submanifold of constant sectional curvature $c$ immersed in an $n$-dimensional complex space form. Then $N$ is totally geodesic in the space form or, if $c=0$, flat.
\end{proposition}

\begin{proof}[Proof of \Cref{thm:product}]
Let $M\subset \R^{m+r}$ be a conformally flat $\widetilde \omega$-calibrated submanifold and let $p\in M$. Then, by \Cref{cor:locflat-new}, there is a (relatively) open neighbourhood $p \in V \subset M$ that is Riemannian isometric to a product $N \times U$ in a neighborhood of $p$, where $N \subset \R^m$ is an $\omega$-calibrated submanifold and $U \subset \R^r$ is an open subset. We will prove that $N$ is flat which will then imply that $V$ and hence $M$ is flat. This will prove that $\widetilde{\omega}$ is conformally rigid.

Suppose first that $r\ge 2$. Since also $n\ge 2$, we have, by Yau's theorem (Proposition \ref{prop:conformallyflat:product}), that $N$ has constant sectional curvature zero. Thus $N$ is a flat submanifold of $\R^m$.

Suppose now that $\omega$ is conformally rigid. Then, by Yau's theorem, $N$ has constant sectional curvature and hence $N$ is conformally flat. Since $N$ is $\omega$-calibrated and $\omega$ is conformally rigid, we conclude that $N$ is flat.

Since a calibration in $\Lambda^2 \mathbb{R}^m$ is equivalent to a symplectic form in a subspace by \Cref{lemma:symplectic-equivalence}, we may apply \Cref{lemm:affinetranslation} and \Cref{cor:faceequivalence} in order to assume that $\omega = \omega_{\symp} \in \Lambda^2 \C^\ell$ in $\C^\ell$ for $1 \leq \ell \leq \lfloor m/2 \rfloor$. If $\ell \geq 2$, then $N \subset \C^{\ell}$ is a K\"ahler submanifold. By Yau's theorem, $N$ has constant sectional curvature. Hence $N$ is Einstein and, by Umehara's theorem \cite{Ume:87}, $N$ is totally geodesic in $\C^\ell$. We conclude that $N$ is affine and, in particular, flat. Similarly, if $\ell = 1$, then $N$ is an open set in $\C$ and therefore flat. Either case, $N$ is flat.

Finally, let $\omega = \omega_\SL$ and $n=m/2$. Again it remains to consider the case $r=1$. Then, by Yau's theorem, $N$ has constant sectional curvature. Since $N$ is calibrated by $\omega_\SL$, it is a totally real submanifold of $\C^n$ (see \Cref{lemma:HL-eq} below for details). Thus, by Ejiri's theorem (Proposition \ref{prop:conformallyflat:product-real}), $N$ is either totally geodesic in $\C^n$ or flat. In either case, $N$ is flat.

We conclude that, in all four cases, the calibration $\widetilde \omega$ is conformally rigid.
\end{proof}

\subsection{Codimension two calibrations and powers of symplectic forms are conformally rigid}
\label{sec:conformal-rigidity-products}

In this section, we prove that the normalized powers of the symplectic form are conformally rigid and therefore inner Möbius rigid by \Cref{thm:Liouville}. We also prove that codimension two calibrations are conformally rigid in dimensions higher than five.
\begin{theorem}\label{thm:kahler}
For $2\le k < \ell$ in $\mathbb{R}^{2\ell}$, the calibration 
\[
    \frac{ \omega_\symp^{ k } }{ k! } \in \Lambda^{2k}( \R^{2\ell} ),
\]
is conformally rigid.
\end{theorem}

\begin{remark}\label{rem:wirtinger}
Recall that $\omega_\symp^{ k } / k!$ is a calibration by Wirtinger's inequality and $\omega_\symp^{ \ell-1 } / (\ell-1)! = \star\omega_{\symp}$, cf. \cite[1.8.2]{Fed:69}.
\end{remark}

\begin{proof}[Proof of Theorem \ref{thm:kahler}]
Let $M$ be a conformally flat $(\omega_\symp^k/k!)$-calibrated submanifold of $\R^{2\ell}$. Then, in particular, $M$ is a conformally flat K\"ahler manifold. 

For $k \geq 3$, conformally flat Kähler manifolds are flat by a result of Yano and Mogi \cite[Theorem 4.1]{Yan:Mog:55}; see also Tanno \cite[p. 501]{Sh:72}. 

For $k=2$, we have that either $M$ is flat or, by Tanno's theorem \cite[p. 501, Theorem]{Sh:72}, that for every $p \in M$, there exists a connected open neighbourhood $W \subset M$ of $p$, $K \ge 0$, and an isometric map $\phi \colon M_1 \times M_2 \rightarrow W$, where $M_1$ and $M_2$ are Kähler surfaces with constant sectional curvature $-K$ and $K$, respectively. Since the symplectic forms on $M_1$ and $M_2$ are simple and the metrics on $M_1, M_2$, and $W$ are compatible with the respective symplectic forms and the almost complex structures, see e.g. \cite[Section 2.1]{McD:Sal:04} for the definitions, it is straight-forward to verify that $\phi$ is a symplectomorphism and, a posteriori, that $D\phi \circ J_{ M_1 \times M_2 } = J_{ W } \circ D\phi$ for the almost complex structures. Thus $\phi$ is, in fact, holomorphic.

Let $p_1\in M_1$. Then the surface $\phi( \left\{p_1\right\} \times M_2 ) \subset M \subset \mathbb{C}^l$ is a Kähler surface of $\mathbb{C}^l$ of constant sectional curvature $K \geq 0$. Since the sectional curvature of Kähler surfaces in $\mathbb{C}^l$ is non-positive by \cite[Chapter IX, Proposition 9.4]{Kob:Nom:96:II}, we conclude that $K=0$. Thus $M$ is flat also for $k=2$.
\end{proof}

\codimMobius*

\begin{proof}
Let $\omega\in \Lambda^{m-2}\R^m$ be a calibration. Then $(-1)^{ n(m-n) }\star \omega \in \Lambda^2 \R^m$ is a calibration by \Cref{lemma:calibration-duality} and, by \Cref{lemma:symplectic-equivalence}, there exist an orthogonal projection $\pi \colon \R^m \to V \subset \R^m$ and an isometric linear isomorphism $\iota \colon V \to \mathbb{R}^{2\ell}$ for which $(-1)^{ n(m-n) }\star \omega$ is equivalent to $( \iota \circ \pi )^{*} \omega_{\symp}$ for the standard symplectic form $\omega_\symp \in \Lambda^2 \R^{2\ell}$ in $\R^{2\ell}$ for $1 \le \ell \le \lfloor m/2 \rfloor$. By \Cref{lemma:calibration-duality}, we have that $\omega$ and $\star ( \iota \circ \pi )^{*}\omega_{\symp}$ are equivalent. So it suffices to prove that $\star ( \iota \circ \pi )^{*}\omega_{\symp}$ is conformally rigid. We consider three separate cases.

Suppose first that $\ell = 1$. Then both $\omega_{\symp}$ and $\star ( \iota \circ \pi )^{*}\omega_{\symp}$ are simple. The conformal rigidity holds in this case. 

Suppose now that $2 \ell = m$. Then the map $\iota \circ \pi \colon \R^m \to \R^{2\ell}$ is an isometric linear isomorphism and 
\begin{equation*}
    \star ( \iota \circ \pi )^{*}\omega_{\symp} = ( \det ( \iota \circ \pi ) )\iota^{*}( \star \omega_{\symp} ) = ( \det ( \iota \circ \pi ) )( \iota \circ \pi )^{*}\left( \frac{ \omega_{\symp}^{ \ell -1  } }{ (\ell - 1)! } \right),
\end{equation*}
where the first equality follows from the definition of the Hodge star and the second from Wirtinger's inequality, cf. \Cref{rem:wirtinger}. Since $m\ge 5$, we have that $\ell-1\ge 2$ in this case. Thus $\star ( \iota \circ \pi )^{*}\omega_{\symp}$ is conformally rigid by \Cref{thm:kahler}.

It remains to consider the cases $5 \leq 2\ell + 1 \leq m$. Let $\pi_1 \colon \R^m \to \R^{2\ell}$ and $\pi_2 \colon \R^m \to \R^{m-2\ell}$ be coordinate projections. A direct computation shows that
\begin{equation*}
    \star \pi_{1}^{*}\omega_{\symp} = \pi_1^{*}\left( \frac{\omega_\symp^{\ell-1}}{(\ell-1)!} \right) \wedge \pi_{2}^{*}\vol_{ \R^{m-2\ell } }.
\end{equation*}
Let $\pi^{\perp} \colon \R^m \to V^{\perp} \subset \R^m$ be the orthogonal projection to the orthogonal complement of $V$ and let $j \colon V^{\perp} \rightarrow \mathbb{R}^{m-2\ell}$ be a linear isometry for which the linear map $L \coloneqq ( \iota \circ \pi, j \circ \pi^{\perp} ) \colon \mathbb{R}^m \to \R^{2\ell} \times \R^{m-2\ell}$ is an orientation-preserving isometry. Then
\begin{align*}
L^{*}\left( \pi_1^{*}\left( \frac{\omega_\symp^{\ell-1}}{(\ell-1)!} \right) \wedge \pi_{2}^{*}\vol_{ \R^{m-2\ell } } \right) = L^*\left( \star \pi_1^*\omega_\symp\right) 
    &= \star L^*\pi_1^*\omega_\symp 
    \\
    &= \star (\iota\circ \pi)^*\omega_\symp.
\end{align*}
It remains to prove that $\star (\iota \circ \pi)^{*}\omega_{\symp}$ is conformally rigid. For $\ell=2$, this follows from (3) in \Cref{thm:product} and for $\ell\ge 3$ from (2) in \Cref{thm:product}; recall that (2) is applicable because of \Cref{thm:kahler}.

We have now exhausted all the possible cases so the claim follows.
\end{proof}

\subsection{Proof of Liouville's theorem in low dimensions}
\label{sec:proof-of-classification}

We are now ready to summarize the results in the form of a proof of the main theorem.

\classification*
\begin{proof}
Suppose first that $n=m-1$. Then $\omega$ is simple, $\Gr(\omega) = \{V\}$ is a singleton, and each $\omega$-calibrated submanifold is contained in a translation of $V$ by \Cref{lemm:affinetranslation}. In particular, $\omega$ is conformally flat. For $n=m-2 \geq 3$, the claim follows from \Cref{thm:codim-2}. Thus it remains to prove the claim for calibrations in $\Lambda^3 \R^6$. 

Let $\omega \in \Lambda^3 \R^6$. By the Morgan's classification theorem \cite[Theorem 4.8]{Mor:85}, either $\Gr(\omega)$ is discrete, $\omega$ is face equivalent to $\omega_\SL$, or there exists a five dimensional subspace $\mathbb{V} \subset \mathbb{R}^6$, an orthogonal projection $\pi \colon \mathbb{R}^6 \rightarrow \mathbb{V}$, and a linear isometry $\iota \colon \mathbb{V} \rightarrow \mathbb{R}^4 \times \mathbb{R}$ such that $\omega$ is equivalent to $( \iota \circ \pi )^{*}( \omega_{\symp} \wedge dt )$.  

In the last case, $\omega$ is conformally rigid by \Cref{thm:codim-2} (and \Cref{lemm:affinetranslation}). Thus it remains to consider the first case. Suppose $\Gr(\omega)$ is discrete and let $M \subset \R^6$ be a connected $\omega$-calibrated submanifold. Then the Gauss map $M \to \Gr(\omega)$, $p \mapsto T_p M$, is constant and so $M$ is calibrated by a simple calibration in $\Lambda^3 \R^6$. Thus, again by \Cref{lemm:affinetranslation}, $M$ is contained in an affine $3$-dimensional subspace of $\R^6$. This completes the proof.
\end{proof}

%%%%%%%%%%%%%%%%%%%%%%%%%%%%%%%%%%%%%%%%%%%%%%%%%%%%%%%%%%%%%%%%%%%%%%%%%%%%%%%
%%%%%%%%%%%%%%%%%%%%%%%%%%%%%%%%%%%%%%%%%%%%%%%%%%%%%%%%%%%%%%%%%%%%%%%%%%%%%%%

\section{Examples of calibrations without the Liouville property}
\label{sec:examples}

\subsection{Special Lagrangians are not inner M\"obius rigid}
\label{sec:Special-Lagrangians}
In the first subsection below, we consider the Lagrangian catenoid example in \Cref{prop:SL-catenoid} in further detail. In the second subsection below, we consider the mapping in \Cref{prop:SL-example-concrete}.

\subsubsection{The Lagrangian catenoid}
Recall that $F = \widehat F \circ \varphi \colon \R^n\setminus \{0\} \to \R^{2n}$, where $\varphi \colon \R^n \setminus \{0\} \to \R\times \bS^{n-1}$ is the map $x\mapsto (\log |x|, x/|x|)$ and $\widehat F \colon \R\times \bS^{n-1}\to \C^n$ is the map $(t,p) \mapsto f(t)p$, where $f \colon \R\to \C$ is the function $t\mapsto (\sinh(nt)-i)^{1/n}$. Here we identify each $z = ( z_1, \dots, z_n ) \in \C^n$ with $(x,y) \in \R^n \times \R^n$ so that $z_j = x_j + i y_j$ for $1 \leq j \leq n$.

Since $\varphi$ is a conformal homeomorphism, it suffices to show that $\widehat F$ is conformal and satisfies the equation $\star \widehat F^*\omega_\SL = \| D\widehat F \|^n$.

The inclusion $\S^{n-1} \subset \R^n$ identifies the tangent space $T_{(t,p)}( \R \times \S^{n-1} )$ with $\R \times T_{p} \S^{n-1}$. Observe also that the real and imaginary parts of $\widehat{F}$ are orthogonal to $\S^{n-1} \subset \R^n$, so the representation of $( D_{ (t,p) }\widehat{F} )^{*} D_{ (t,p) }\widehat{F}$, where $*$ refers to the Hermitean conjugate of the differential, in the standard basis has a $2 \times 2$-block form with zero off-diagonal blocks. In fact, as
\begin{equation*}
    \partial_t \widehat{F}(t,p)
    =
    \frac{ \cosh(nt) }{ \sinh(nt) - i } \widehat{F}(t,p),
\end{equation*}
the top-left most block of $( D\widehat{F} )^{*} D\widehat{F}$ is
\begin{equation*}
    \frac{ \cosh^2(nt) }{ 1 + \sinh^2(nt) }| \widehat{F}(t,p) |^2
    =
    \cosh^{ 2/n }( nt ).
\end{equation*}
Similarly, the bottom-right block is an $(n-1)\times(n-1)$ matrix $\cosh^{ 2/n }(t) I_{n-1}$ for the identity matrix $I_{n-1}$ on $T_p \S^{n-1}$. So
\begin{equation*}
    ( ( D\widehat{F} )^{*} D\widehat{F} )_{ (t,p) } = \cosh^{ 2/n }(t) I_{n} \quad\text{in $\R \times T_{p} \S^{n-1}$}. 
\end{equation*}
Hence $\widehat{F}$ is a conformal embedding with $\| D\widehat{F} \|^n = \cosh(nt)$. To get that $\widehat{F}$ is calibrated by a Special Lagrangian, let $e_2, \dots, e_n \in T_{p}\S^{n-1}$ be orthonormal such that $( p, e_2, \dots, e_n )$ is oriented in the standard way in $\R^n$. We denote $v_1 = \partial_t \widehat{F}(t,p) = ( \cosh(nt) ) / ( \sinh(nt)-i ) f(t) (p,0)$ and $v_j = ( D \widehat{F} )_{ (t,p) } e_j = f(t) ( e_j, 0 )$ for $2 \leq j \leq n$. Then
\begin{align*}
    v_1 \wedge v_2 \wedge \dots \wedge v_{n}
    &=
    \frac{ \cosh(nt) }{ ( \sinh(nt)-i ) } ( f(t) )^n
    ( p, 0 ) \wedge ( e_2, 0 ) \wedge \dots \wedge ( e_n, 0 )
    \\
    &=
    \cosh(nt) ( p, 0 ) \wedge ( e_2, 0 ) \wedge \dots \wedge ( e_n, 0 ).
\end{align*}
Notice also that $J(p,0) = ( 0, p )$ and $J ( e_j, 0 ) = ( 0, e_j )$ for $2 \leq j \leq n$ for the almost complex structure $J$ on $\C^n$. So the plane $\mathbb{V}_p$ spanned by $( (p,0), ( e_2, 0), \dots, ( e_n, 0) )$ is a Lagrangian plane. In fact, $$\omega_{\SL}( (p,0), (e_2,0), \ldots, ( e_n, 0 ) ) = 1$$ by the following argument. Since the Special Unitary group acts on the set of maximal planes of $\omega_{\SL}$, we consider $O \in \SO(n)$ which maps $( p, e_2, \dots, e_n )$ to the standard basis on $\R^n$ and thus $T= O \times O \colon \R^n \times \R^n \to \R^n\times \R^n$ in $\mathrm{SU}(n)$ maps $\mathbb{V}_p$ to the Special Lagrangian plane span by the first $n$-coordinate vectors in $\C^n = \R^n \times \R^n$. So
\begin{align*}
    \star \widehat{F}^{*} \omega_{\SL}(t,p)
    &=
    \omega_{\SL}( v_1, v_2, \ldots, v_n )
    =
    \cosh(nt)
    =
    \| D \widehat{F} \|^n(t,p).
\end{align*}
We conclude $\star \widehat{F}^*\omega_{\SL} = \| D \widehat{F} \|$ and \Cref{prop:SL-catenoid} follows.

\subsubsection{Mapping through a torus}
We begin by deconstructing the mapping $h \colon \R^{n+1} \to \C^n$ in the statement of \Cref{prop:SL-example-concrete}.  Let $\bS = \bS^{2n-1} \subset \C^n=\R^{2n}$ be the Euclidean unit sphere and $\bT = \bS^1(\frac{1}{\sqrt{n}})\times \cdots \times \bS^1(\frac{1}{\sqrt{n}}) \subset \bS \subset \C^n$ a torus.  

We introduce two auxiliary mappings. Let $\varphi \colon \R^n \times \R \to \bT \times \R$ be the locally isometric covering map,
\[
(y_1,\ldots, y_n,y_{n+1}) \mapsto \left( \frac{1}{\sqrt{n}}e^{i\sqrt{n} y_1}, \ldots, \frac{1}{\sqrt{n}}e^{i\sqrt{n} y_n},y_{n+1}\right),
\]
and let $E \colon \bS \times \R \to \C^n$ be the conformal (exponential) map $(z,t) \mapsto e^tz$. Since $\varphi$ is a local isometry and $E$ is conformal, we have that $E\circ \varphi$ is weakly conformal and $\norm{D(E\circ \varphi)} = \norm{DE} \circ \varphi$. Moreover, $\norm{DE(p,t)} = e^t$ for $(p,t)\in \bS \times \R$. Since 
\begin{align*}
(E \circ \varphi)(y_1,\ldots, y_n,t) 
&= e^t \left( \frac{1}{\sqrt{n}}e^{i\sqrt{n} y_1}, \ldots, \frac{1}{\sqrt{n}}e^{i\sqrt{n} y_n}\right)
\\
&= \frac{1}{\sqrt{n}} e^t \left( e^{i\sqrt{n} y_1}, \ldots, e^{i\sqrt{n} y_n}\right)
\end{align*}
for $(y_1,\ldots, y_n, t)\in \R^n\times \R$, we have that 
\[
h =  i^{-(n-1)} E\circ \varphi \colon \R^{n+1} \to \C^n.
\]
In particular, $h$ is conformal and $\norm{Dh(y_1,\ldots, y_n,t)} = e^{t}$ for $(y_1,\ldots, y_n,t) \in \R^{n+1}$.

Before defining the mapping $O$ and completing the proof of \Cref{prop:SL-example-concrete}, we analyze the mapping $h$ further. Since $E \circ \varphi$ is weakly conformal, so is $h$. Furthermore $\norm{Dh(y,t)} =  e^{t}$ for $(y,t) \in \R^n \times \R$. We also have the following lemma. For the lemma, we denote
\begin{align*}
    \eta
    &\coloneqq
    \sum_{ j = 1 }^{ n }
        (-1)^{ n-j }
        dy_1 \wedge \dots \wedge \widehat{ dy_j } \wedge \dots \wedge dy_{n} \wedge dt 
    \in
    \Lambda^{n}( \mathbb{R}^n \times \mathbb{R} ).
\end{align*}
Let also $S \colon \R^n \to \R$ be the sum $(y_1,\ldots, y_n) \mapsto \sum_{j=1}^n y_j$.

\begin{lemma}
\label{lemma:SL-h-properties}
For the mapping $h = i^{-(n-1)} E\circ \varphi \colon \R^{n+1} \to \C^n$, we have
\[
(h^*\omega_\SL)|_{\ker S\times \R} = \frac{e^{nt}}{\sqrt{n}} \eta|_{\ker S \times \R}.
\]
\end{lemma}
\begin{proof}
We observe first that
\begin{align*}
&h^{*}(dz_1 \wedge \cdots \wedge dz_n) = i^{ -(n-1) } d\left( \frac{1}{\sqrt{n}} e^t e^{i\sqrt{n}y_1} \right) \wedge \dots \wedge d\left( \frac{1}{\sqrt{n}} e^{t} e^{i\sqrt{n}y_n } \right) \\
%&= n^{ -\frac{n}{2} } i^{ -(n-1) } \left( ( e^{t} e^{i \sqrt{n} y_1 } dt + i \sqrt{n} e^{t} e^{ i\sqrt{n} y_1 } dy_1 ) \wedge \dots \wedge ( e^{t} e^{i\sqrt{n} y_n } dt + i \sqrt{n} e^{t} e^{ i\sqrt{n} y_n } dy_n ) \right) \\
%&= n^{ -\frac{n}{2} } i^{ -(n-1) } e^{nt} \left( ( e^{i \sqrt{n} y_1 } dt + i \sqrt{n} e^{ i\sqrt{n} y_1 } dy_1 ) \wedge \dots \wedge ( e^{i\sqrt{n} y_n } dt + i \sqrt{n} e^{ i\sqrt{n} y_n } dy_n ) \right) \\
&= n^{ -\frac{n}{2} } i^{ -(n-1) } e^{nt} \left( e^{i\sqrt{n}y_1} ( dt + i \sqrt{n}  dy_1 ) \wedge \dots \wedge e^{i\sqrt{n}y_n}( dt + i \sqrt{n} dy_n ) \right) \\
&= n^{ -\frac{n-1}{2} } i^{ -(n-1) } e^{nt} e^{i\sqrt{n} S(y)} \left( ( dt + i \sqrt{n}  dy_1 ) \wedge \dots \wedge ( dt + i \sqrt{n} dy_n ) \right).
\end{align*}
Thus, on $\ker S\times \R$, 
\begin{align*}
h^*(dz_1 \wedge \cdots \wedge dz_n)
&= n^{ -\frac{n}{2} } i^{ -(n-1) } e^{nt} \left( ( dt + i \sqrt{n}  dy_1 ) \wedge \dots \wedge ( dt + i \sqrt{n} dy_n ) \right) \\
&= n^{ -\frac{n}{2} } i^{ -(n-1) } e^{nt} \left( (i\sqrt{n})^n dy_1\wedge \cdots \wedge dy_n + (i\sqrt{n})^{n-1} \eta\right)\\
&= \frac{e^{nt}}{\sqrt{n}} \left( i\sqrt{n} dy_1\wedge \cdots \wedge dy_n + \eta\right).
\end{align*}
Hence, on $\ker S\times \R$, 
\[
h^*\omega_\SL = h^*\Re (dz_1 \wedge \cdots \wedge dz_n) = \Re\, h^*(dz_1 \wedge \cdots \wedge dz_n) = \frac{e^{nt}}{\sqrt{n}} \eta
\]
as claimed.
\end{proof}

We are now ready to prove \Cref{prop:SL-example-concrete}.

\begin{proof}[Proof of \Cref{prop:SL-example-concrete}]
Since $\eta$ is simple and $\star \eta = - \left( dy_1 + \cdots + dy_n\right)$, we have that 
\[
\norm{\eta}_{\comass} 
= \norm{\star \eta}_\comass 
= \norm{dy_1 + \cdots + dy_n}_{HS} 
= \sqrt{n}
\]
and that 
\[
\Gr\left(\frac{1}{\sqrt{n}} \eta \right) = \{ \langle v \rangle^\bot \},
\] 
where $v\in \R^{n+1}$ is the vector $v = e_1+ \cdots + e_n$. 

Let now $O \colon \R^n \to \langle v \rangle^\bot$ be an isometry for which $\star O^*\eta>0$. Then $O^*\eta = \sqrt{n} \vol_{\R^n}$. Let also $h \colon \R^{n+1} \to \C^n$ be as in Lemma \ref{lemma:SL-h-properties}. We define $F \colon \R^n \to \C^n$ by the formula
\[
F = h \circ O.
\]
Now $F$ is conformal and $\norm{DF(x_1,\ldots, x_n)} = e^{x_n}$ for each $(x_1,\ldots, x_n)\in \R^n$. Since $O^*\eta = \sqrt{n} \vol_{\R^n}$, we have, by Lemma \ref{lemma:SL-h-properties}, that
\[
\star F^*\omega_\SL 
= O^* h^*\eta = O^* \left(\frac{e^{nt}}{\sqrt{n}} \eta \right) 
= e^{nx_n} \vol_{\R^n} 
= \norm{DF}^n \vol_{\R^n}.
\]
The claim follows. 
\end{proof}

%%%%%%%%%%%%%%%%%%%%%%%%%%%%%%%%%%%%%%%%%%%%%%%%%%%%%%%%%%%%%%%%%%%%%%%%%%%%%%%
%%%%%%%%%%%%%%%%%%%%%%%%%%%%%%%%%%%%%%%%%%%%%%%%%%%%%%%%%%%%%%%%%%%%%%%%%%%%%%%

\subsection{Associative and Cayley calibrations are not inner M\"obius rigid}
\label{sec:special-cases}

Based on the non-rigidity of the Special Lagrangian, we obtain that neither the associative calibration $\omega_\assoc \in \Lambda^3 \R^7$ nor the Cayley calibration $\omega_\Cayley \in \Lambda^4 \R^8$ are inner M\"obius rigid. We refer to Joyce \cite[Sections 12.2.1 and 12.4.2]{Joy:07} for the terminology and more detailed discussion on the calibrations above; see also Harvey and Lawson \cite[Section III.1]{Har:Law:82}.

For the definitions below, let $\omega_{\symp}$ and $\widetilde{\omega}_{\symp}$ denote the standard symplectic forms on $\R^6$ and $\R^8$, respectively. We also let $\omega_{\SL} = \Re( dz_1 \wedge dz_2 \wedge dz_3 )$ and $\widetilde{\omega}_{\SL} = \Re( dz_1 \wedge dz_2 \wedge dz_3 \wedge dz_4 )$ denote the Special Lagrangian on $\R^6$ and $\R^8$, respectively.

We consider the coordinate projection $\pi \colon \R^7 \to \R^6$, $( t, x_1, \dots, x_6 ) \mapsto ( x_1, \dots, x_6 )$. Then
\begin{align*}
    \omega_\assoc &=  dt \wedge \pi^*\omega_{\symp} + \pi^*\omega_{\SL} 
    \in \Lambda^3 \R^7, 
    \\
    \omega_{\coassoc} 
    &= \star \omega_\assoc 
    = 
    \frac{1}{2} \pi^* \left(  \omega_\symp \wedge \omega_\symp \right) - dt \wedge \pi^*( \star \omega_\SL ),
\quad \text{and} \\
\omega_\Cayley 
&= \frac{1}{2}\left( \widetilde{\omega}_{\symp} \wedge \widetilde{\omega}_{\symp}\right) + \widetilde{\omega}_{\SL} 
\in \Lambda^4 \R^8
\end{align*}
are the associative and coassociative on $\R^7$ and Cayley form on $\R^8$, respectively.

Calibrations $\omega_\assoc$ and $\omega_\Cayley$ are not inner Möbius rigid. For these observations, we recall a fact from Harvey and Lawson \cite{Har:Law:82}.
\begin{lemma}
\label{lemma:HL-eq}
Let $\omega_\symp \in \Lambda^2 \R^{2n}$ be the symplectic form and $\omega_\SL \in \Lambda^n \R^{2n}$ the Special Lagrangian for $n \geq 3$. Then $\omega_\symp|_V = 0$ for $V\in \Gr(\omega_\SL)$.
\end{lemma}
\begin{proof}
By \cite[Equation (1.2), p. 87]{Har:Law:82}, we have, for $P\in \Gr(\omega_\SL)$ and $u\in P$, that $i u \bot P$. Thus $\omega_\symp|_V = 0$ for each $V\in \Gr(\omega_\SL)$.
\end{proof}
\begin{corollary}\label{cor:assoc:cayley:notrigid}
The calibrations $\omega_\assoc\in \Lambda^3 \R^7$ and $\omega_\Cayley \in \Lambda^4 \R^8$ are not inner Möbius rigid.
\end{corollary}
\begin{proof}
Let $F\colon \Omega \to \R^6$ be a conformal $\omega_\SL$-curve and let $H = \iota \circ F \colon \Omega \to \R \times \R^6$ be a map, where $\iota \colon \R^6 \to \R^7$ is the standard inclusion $x \mapsto (0,x)$. Then, by \Cref{lemma:HL-eq}, $H^*\omega_\assoc = F^*\omega_\SL$. Since $\norm{DH} = \norm{DF}$, we conclude that $H$ is a conformal $\omega_\assoc$-curve. In particular, $\omega_\assoc$ is not inner M\"obius rigid.

Regarding the Cayley calibration, let $F\colon \Omega \to \R^8$ be a conformal $\widetilde{\omega}_\SL$-curve that is not inner Möbius. Then, again by \Cref{lemma:HL-eq}, $F^*\omega_\Cayley = F^*\widetilde{\omega}_\SL$. Thus $F$ is a conformal $\omega_\Cayley$-curve. Since $\widetilde{\omega}_{\SL}$ is not inner Möbius rigid, neither is the Cayley calibration.
\end{proof}

\begin{remark}
To our knowledge, it is an open question whether there are conformal $\omega_{\assoc}$-curves which are not $\pi^*\omega_\SL$-curves or whether there are conformal $\omega_\Cayley$-curves which are not $\pi^*\omega_\SL$-curves. More generally, we are not aware of a classification of conformal curves associated to calibrations $\omega_\assoc\in \Lambda^3 \R^7$, $\omega_\coassoc \in \Lambda^4 \R^7$, and $\omega_\Cayley \in \Lambda^4 \R^8$. In particular, we do not know if $\omega_{\coassoc}$ is inner Möbius rigid.
\end{remark}

\section{Space of inner M\"obius rigid calibrations and quasiregular curves}
%\section{Stability of quasiregular curves of small distortion}
\label{sec:HPP}

Consider a calibration $\omega \in \Lambda^n \mathbb{R}^m$, for $3 \leq n \leq m$, and for $t \geq 0$, and for each open set $U \subset \mathbb{R}^n$, let $\mathcal{F}^{\omega}_{t}(U)$ denote the collection $(1+t)$-quasiregular $\omega$-curves defined on $U$. That is, each element is a mapping $F \colon U \to \mathbb{R}^m$ satisfying $F \in W^{1,n}_{\loc}( U, \mathbb{R}^m )$ and $\| DF \|^n \leq (1+t) \left(\star F^{*}\omega\right)$. Finally, let $\mathcal{F}^{\omega}_t$ denote the collection of every $(1+s)$-quasiregular $\omega$-curve $F \colon B_1 \to \mathbb{R}^m$ mapping from the unit ball $B_1 \subset \mathbb{R}^n$ into the unit ball of $\mathbb{R}^m$ for $0 \leq s \leq t$.

\subsection{Modulus of continuity for quasiregular curves}
We recall that by \cite[Theorem 6.8]{Iko:23}, the elements of $\mathcal{F}^{\omega}_t(U)$ are Hölder-continuous with exponent $(1+t)^{-1}$. Furthermore, by a standard Caccioppoli argument, cf. \cite[Lemma 6.1]{Onn:Pan:21}, the diameter of the image controls the energy of a quasiregular curve. We record these observations in the following proposition.

\begin{proposition}\label{prop:localcontinuity}
Let $\omega \in \Lambda^n \mathbb{R}^m$ be a calibration for $3 \leq n \leq m$, and let $F \colon U \to \mathbb{R}^m$ be a $(1+t)$-quasiregular $\omega$-curve for $t \geq 0$. Then, for $B( x_0, r ) \subset U$, we have
\begin{equation*}
    | F(x) - F(y) |
    \leq
    C(n,t)
    \left( 
        \int_{ B( x_0, r ) } \|DF\|^{n} \,d\mathcal{H}^n
    \right)^{ \frac{1}{n} }
    r^{ - \frac{1}{1+t} }
    | x - y |^{ \frac{1}{1+t} }
\end{equation*}
for $x, y \in B( x_0, r/2 )$.

Moreover, for $B( x_0, 2r ) \subset \Omega$, we have that
\begin{equation*}
    \left( \int_{ B( x_0, r ) } \|DF\|^n \,d\mathcal{H}^n \right)^{ \frac{1}{n} }
    \leq
    C(n) (1+t)
    \diam F(B(x_0,2r)).
\end{equation*}
\end{proposition}
\begin{proof}
Let $T(x) = x_0 + rx$ be the affine map taking the unit ball $B$ to the ball $B(x_0,r)$. Then, by \cite[Theorem 6.8]{Iko:23}, the $(1+t)$-quasiregular $\omega$-curve $f = F \circ T$ satisfies the modulus of continuity estimate for $x_0 = 0$ and $r = 1$. The general modulus of continuity estimate is immediate.

Regarding the control on the energy, we denote $K = 1+t$ for simplicity and let $\varphi \in C^\infty_0(B(x_0,2r))$ be a non-negative function for which $\varphi \equiv 1$ in $B(x_0,r)$ and $\norm{ D\varphi} \le 3/r$. Since $\omega\in \Lambda^n \R^m$ is a calibration, we have, by \cite[Lemma 6.1]{Onn:Pan:21}, that
\begin{align*}
\int_{ B(x_0,r)} \norm{DF}^n \,d\mathcal{H}^n 
&\le K\int_{B(x_0,r)} F^*\omega
\le K \int_{B(x_0,2r)} \varphi^n F^*\omega \\
&\le n^n K^{n} \int_{B(x_0,2r)} |D \varphi|^n(x) |F(x)-F(x_0)|^n  \,d\mathcal{H}^n(x) \\ 
&\le n^n K^n \int_{B(x_0,2r)} |D \varphi|^n \,d\mathcal{H}^n ( \diam FB(x_0,2r) )^{n} \\
&\le C(n) K^n ( \diam FB(x_0,2r) )^n,
\end{align*}
where $C(n)$ is a constant depending only on $n$.
\end{proof}

%%%%%%%%%%%%%%%%%%
%%%%%%%
%%%%%%%
%%%%%%%%%%%%%%%%%%

\subsection{Compactness properties of quasiregular curves}
As an application of \Cref{prop:localcontinuity}, we establish the following compactness result.
\begin{proposition}\label{prop:normalfamily}
Let $U \subset \mathbb{R}^n$ be a domain and let $\omega_l \in \Lambda^n \mathbb{R}^m$ be a sequence of calibrations, for $3 \leq n \leq m$, and consider a sequence of $( 1+t_l)$-quasiregular $\omega_l$-curves $F_l \colon U \rightarrow \mathbb{R}^m$. Suppose $\sup t_l = T < \infty$ and $\sup \| F_l \|_\infty < \infty$. Then there exists a calibration $\omega \in \Lambda^n \mathbb{R}^m$, $t \in [0,T]$ and a $( 1+t )$-quasiregular $\omega$-curve $F \colon U \to \mathbb{R}^m$ such that a subsequence of $( F_l, \omega_l, t_l )$ converges to $( F, \omega, t )$. More precisely, there exists a sequence $l_j \rightarrow \infty$ for which $F_{ l_j }$ converges to $F$ uniformly on compact sets of $U$, $\omega_{ l_j }$ to $\omega$ in $\Lambda^{n} \mathbb{R}^m$ and $t_{ l_j }$ to $t$ in $\mathbb{R}.$
\end{proposition}
\begin{proof}
By \Cref{prop:localcontinuity}, we deduce that the family $( F_l|_{ K } )_{ l = 0 }^{ \infty }$ is normal for every compact set $K \subset U$ by Arzelà--Ascoli theorem. Hence we obtain a continuous map $F \colon U \to \mathbb{R}^m$ that is a locally uniform limit of a subsequence $( F_{l_j} )_{ j = 1 }^{ \infty }$. Consider such a limit and a subsequence.

By passing to a further subsequence and relabeling, we lose no generality in assuming that $t_j \rightarrow t \in [0,T]$ and $\omega_{l_j} \rightarrow \omega$ in $\Lambda^{n} \mathbb{R}^m$. In particular, for every $1 > \epsilon > 0$, there exists $j_0 \in \mathbb{N}$ such that every $j \geq j_0$ satisfies
\begin{align*}
    t_{l_j} \leq \epsilon + t
    \quad\text{and}\quad
    \|\omega_{l_j} - \omega\|_{ \comass }
    \leq
    \frac{ \epsilon }{ 1 + \epsilon + t }.
\end{align*}
Therefore, by Hadamard inequality, we deduce that
\begin{align*}
    \star F^{*}_{ l_j }\omega
    &\geq
    -
    \frac{ \epsilon }{ 1 + \epsilon + t }
    \| DF_{ l_j } \|^n
    +
    \star F^{*}_{ l_j }\omega_{ l_j }
    \geq
    \left( - \frac{ \epsilon }{ 1 + \epsilon + t } + \frac{ 1 }{ 1 + t_j } \right) \| DF_{ l_j } \|^n
    \\
    &\geq
    \left( 1 - \epsilon \right)
    \frac{ 1 }{ 1 + \epsilon + t }
    \| DF_{ l_j } \|^n.
\end{align*}
So $F_{ l_j }$ is a $( 1 + \epsilon + t )/( 1 - \epsilon )$-quasiregular $\omega$-curve for $j \geq j_0$. This is true for the locally uniform limit $F$ as well, cf. \cite[Theorem 4.1]{Pan:20}. Since $\epsilon > 0$ was arbitrary, the limit $F$ is a $(1+t)$-quasiregular $\omega$-curve.
\end{proof}

\subsection{Iwaniec function}

For $3 \le n \le m$, recall that $\cC_{n}( \R^m ) \subset \Lambda^n \R^m$ is the space of all calibrations in $\Lambda^n \R^m$ and $\mathcal{I}_n( \R^m ) \subset \mathcal C_{n}( \R^m )$ is the subspace of all inner Möbius rigid calibrations on $\Lambda^{n}\mathbb{R}^m$.

Let $\tau \colon \cC_{n}( \R^m ) \times [0,\infty) \to [0,2]$, $(\omega, t) \mapsto \tau^\omega(t)$, be the function given by the formula 
\begin{equation}\label{eq:functionfunction}
    \tau^{\omega}(t)
    \coloneqq
    \sup_{ F \in \mathcal{F}_t^{\omega} }
    \min_{ \substack{ G \colon B_1 \to \R^m \\ \text{inner Möbius} \\ F(0) = G(0) } } d( F, G ),
\end{equation}
where we use the notation 
\begin{equation*}
    d( F, G ) \coloneqq \sup_{ x \in B_1 } | 1 - |x| | | F(x) - G(x) |.
\end{equation*}
We call the function $\tau$ an \emph{Iwaniec function} as it was used in \cite{Iwa:87} to study the distance of quasiregular mappings of small distortion to M\"obius transformations, that is, in the case $n=m$ and $\omega = \vol_{\R^n}$. Note that $\norm{\tau}_\infty \le 2$, since isometric affine $\omega$-curves are inner Möbius, and that the minimum in \eqref{eq:functionfunction} is justified by \Cref{prop:normalfamily}.

Iwaniec function identifies inner M\"obius rigid calibrations. More precisely, a calibration $\omega\in \cC_n(\R^m)$ is inner Möbius rigid if and only if $\tau^{\omega}(0) = 0$. Indeed, $d( F, G ) = 0$ for $F \in \mathcal{F}_t^{\omega}$ and inner Möbius $G$ if and only if $F \equiv G$ and thus $F$ is a $(1+t)$-quasiregular $\omega$-curve that is inner Möbius. The Iwaniec function also has the following pointwise continuity property for inner M\"obius rigid calibrations.

\begin{lemma}\label{lemm:confrigid:closed}
For every calibration $\omega \in \mathcal{I}_{n}( \R^m )$, 
\[
\lim_{ t \rightarrow 0^{+} } \tau^{\omega}(t) = 0, \quad\text{for $3 \le n \le m$.}
\]
\end{lemma}

\begin{proof}
By definition there are $t_l \in (0, 2^{-l})$ and $F_l \in \mathcal{F}^{ \omega }_{ t_l }$ such that
\[
    | \tau^{\omega}( t_l ) - \lim_{ t \rightarrow 0^{+} } \tau^{\omega}( t ) | \leq 2^{-l}
\]
and
\[
| \tau^{\omega}( t_l ) - \min_{ \substack{G \in \mathcal{F}^{\omega}( B_1 ) \\ F_l(0) = G(0)} }
    d( F_l, G ) | \leq 2^{-l}.
\]

\Cref{prop:normalfamily} implies that up to passing to a subsequence and relabeling, we lose no generality in assuming that $t_l \rightarrow 0$ and $F_l$ converges locally uniformly to a conformal $\omega$-curve $F \colon B_1 \to \mathbb{R}^m$. If $H_l = F - F(0) + F_{l}(0)$, then $H_l$ is an inner Möbius curve with $H_l(0) = F_l(0)$ as $\omega$ is inner Möbius rigid. Consequently,
\begin{equation*}
    \lim_{ t \rightarrow 0^{+} } \tau^{\omega}( t )
    \leq
    2^{ 1 - l }
    +
    \sup_{ x \in B_1 } | 1 - |x| | | F_l(x) - H_l(x) |.
\end{equation*}
If $\epsilon > 0$, by locally uniform convergence of $F_l$ to $F$, there exists $l_0 \in \mathbb{N}$ such that for every $l \geq l_0$, we have $2^{1-l} < \epsilon/2$ and
\begin{equation*}
    | F_l(x) - H_l(x) | < \frac{ \epsilon }{ 2 } \quad\text{if $x \in B_1$ satisfies $| 1 - |x| | \geq \epsilon/2$.}
\end{equation*}
Moreover, if $x \in B_1$ satisfies $| 1 - |x| | < \epsilon/2$, then
\begin{equation*}
    | 1 - |x| | | F_l(x) - H_l(x) | < \epsilon.
\end{equation*}
So for $l \geq l_0$, we have
\begin{equation*}
    d( F_l, H_l ) = \sup_{ x \in B_1 } | 1 - |x| | | F_l(x) - H_l(x) | \leq \epsilon.
\end{equation*}
Consequently, $\lim_{ t \rightarrow 0^{+} } \tau^{\omega}( t ) \leq \epsilon$ and so $\lim_{ t \rightarrow 0^{+} } \tau^{\omega}( t ) = 0$.
%by the arbitrariness of $\epsilon > 0$.
\end{proof}

%%%%%%%%%%%%%%%%%%%%%%%%%%%%%%%%%
%%%%%%%%%%%%%%%%%%%%%%%%%%%%%%%%%%

\subsection{The space of inner Möbius rigid calibrations}
\label{sec:HPP:structure}

In this section, we prove the density of $\mathcal{I}_n(\R^m)$ in $\cC_n(\R^m)$ and use the Iwaniec function to show that $\mathcal{I}_n(\R^m)$ is a $G_\delta$ set in $\cC_n(\R^m)$. We begin with the result on density.

\begin{lemma}
\label{lemma:confrigid:supremum-fail}
For $3 \leq n \leq m$, the set $\mathcal{I}_n( \R^m )$ is dense in $\mathcal{C}_n( \R^m )$.
\end{lemma}

\begin{proof}
We may assume that that there exists $\omega \in \mathcal{C}_n( \R^m ) \setminus \mathcal{I}_n( \R^m )$. By \Cref{cor:applicability}, $\omega = \widetilde{\omega} + \epsilon$, where $\widetilde{\omega}$ is a conformally rigid calibration and $\omega_t = \widetilde{\omega} + t\epsilon$ is a calibration equivalent to $\widetilde{\omega} = \omega_0$ for every $0 \leq t < 1$ and thus conformally rigid. We deduce from \Cref{thm:Liouville} that $\omega_t$ is inner Möbius rigid for $0 \leq t < 1$. The density of $\mathcal{I}_n( \R^m )$ in $\mathcal{C}_n( \R^m )$ follows.
\end{proof}

To show that $\mathcal{I}_n(\R^m)$ is a $G_\delta$ set in $\cC_n(\R^m)$, we prove that $\tau$ is upper semicontinuous.

\begin{lemma}\label{lemm:distortionfunction}
For $3\le n \le m$, the function $\tau \colon \cC_{n}( \R^m ) \times [0,\infty) \to [0,2]$ is upper semicontinuous. In particular, $\left\{ \omega \in \cC_{n}(\R^m) \colon \tau^{\omega}(t) < \epsilon \right\}$ is open in $\cC_n( \R^m )$ for every $\epsilon \geq 0$ and $t \geq 0$.
\end{lemma}
\begin{proof}
Consider $t = \lim_{ l \rightarrow \infty } t_l$ and $\omega = \lim_{ l \rightarrow \infty } \omega_l$ in $[0,\infty) \times \cC_{n}( \R^m )$. By passing to a subsequence, we may assume that
\begin{equation*}
    \lim_{ j \rightarrow \infty } \tau^{ \omega_{ l_j } }( t_{ l_j } )
    = 
    \limsup_{ l \rightarrow \infty } \tau^{ \omega_{ l } }(t_{ l }).
\end{equation*}
By passing to a further subsequence and relabeling, we find $F_{l_j} \in \mathcal{F}^{ \omega_{l_j} }_{ t_{l_j} }$, $F \in \mathcal{F}^{ \omega }_{ t }$, and inner Möbius $G_{l_j}, G \colon B_1 \to \R^m$ satisfying $F_{l_j}(0) = G_{l_j}(0)$, $F( 0 ) = G(0)$, and the properties
\[
d( F_{ l_j }, G_{ l_j } )
    =
    \min_{ \substack{ \widetilde{G} \colon B_1 \to \R^m \\ \text{inner Möbius} \\ F_{ l_j }(0) = \widetilde{G}(0) } } d( F_{ l_j }, \widetilde{G} );
\]
\[
| \tau^{ \omega_{ l_j } }(t_{l_j}) - d( F_{ l_j }, G_{ l_j } ) | \leq 2^{-j};
\]
and
\[
    d( F_{ l_j }, F ) + d( G_{ l_j }, G ) \leq 2^{-j}.
\]
To obtain these properties, we apply the compactness result, \Cref{prop:normalfamily}, and the convergence of the function $d( A_j, B_j )$ to zero if $( \sup_{ x \in B_1 }|A_j - B_j| )_{ j = 1 }^{ \infty }$ is a bounded sequence and $( | A_j - B_j | )_{ j = 1 }^{ \infty }$ converges locally uniformly to zero in $B_1$.

Let now $A_{ j } \coloneqq ( G_{ l_j } - F_{ l_j }( 0 ) ) + F(0)$. Then as $F_{ l_j }$ converges locally uniformly to $F$, we deduce that
\begin{equation*}
    \lim_{ j \rightarrow \infty } | d( A_j, F ) - d( G_{ l_j }, F_{ l_j } ) | = 0,
\end{equation*}
so
\begin{equation*}
    \limsup_{ l \rightarrow \infty } \tau^{ \omega_{ l } }(t_l)
    =
    \lim_{ j \rightarrow \infty }
    d( A_j, F )
    \geq
    \min_{ \substack{ \widetilde{G} \colon B_1 \to \R^m \\ \text{inner Möbius} \\ F(0) = \widetilde{G}(0) } } d( F, \widetilde{G} ).
\end{equation*}
Now if $\widetilde{G}$ realizes the above minimum for $d( F, \widetilde{G} )$, we consider $H_j = ( \widetilde{G} - F(0) ) + F_{ l_j }( 0 )$ and deduce that
\[
    \lim_{ j \rightarrow \infty }
    \left| d( F, \widetilde{G} ) - d( F_{ l_j }, H_j ) \right| 
    =
    0
\]
and
\[
    d( F_{ l_j }, H_j )
    \geq
    \min_{ \substack{ \widetilde{G} \colon B_1 \to \R^m \\ \text{inner Möbius} \\ F_{ l_j }(0) = \widetilde{G}(0) } } d( F_{ l_j }, \widetilde{G} )
    =
    d( F_{ l_j }, G_{ l_j } ).
\]

Combining the identities gives that
\begin{align*}
    \limsup_{ l \rightarrow \infty } \tau^{ \omega_l }(t_l)
    =
    d( F, \widetilde{G} )
    \leq
    \tau^{ \omega }(t).
\end{align*}
The upper semicontinuity follows.
\end{proof}

We are now ready to summarize the proof of Theorem \ref{thm:compactnessproperty}.

\compactness*

\begin{proof}
Since
\[
\mathcal{I}_{n}( \R^m ) = \bigcap_{ \epsilon > 0 } \left\{ \omega \in \mathcal{C}_{n}( \R^m ) \mid \tau^{\omega}(0) < \epsilon \right\},
\]
Lemma \ref{lemm:distortionfunction} yields that $\mathcal{I}_n(\R^m)$ is a $G_\delta$ set. The claimed density follows from 
\Cref{lemma:confrigid:supremum-fail}. In case of a strict inequality $\mathcal{I}_n( \R^m ) \subset \mathcal{C}_n( \R^m )$, the density implies non-compactness of $\mathcal{I}_n( \R^m )$.
\end{proof}

%%%%%%%%%%%%%%%%%%%%%%%%%%%%%%%%%%%%%%%%%%%%%%%%%%%%%%%%%%%%%%%%%%%%%%%%%%%%%%%

\subsection{Stability of quasiregular curves}
In this last section, we prove \Cref{thm:localinjec}. The main ingredient in the proof is the following proposition.

\begin{proposition}
\label{prop:HPP-revisited}
Let $3 \le n \le m$ and let $\omega \in \Lambda^n \R^m$ be an inner M\"obius rigid calibration. Then, for $H>1$, there exist $\rho = \rho(H) \in ( 0, 1/4 ]$ and $\varepsilon=\varepsilon(\omega,H)>0$ for which every $( 1 + \varepsilon )$-quasiregular $\omega$-curve is $(H, \rho )$-weakly quasisymmetric. In particular, every nonconstant $( 1 + \varepsilon )$-quasiregular $\omega$-curve is a topological immersion and a discrete map.
\end{proposition}
\Cref{prop:HPP-revisited} is a generalization of \cite[Theorem 1.6]{Hei:Pan:Pry:23}, where the authors consider a calibration $\omega \in \Lambda^n \R^m$ for which $\Gr( \omega )$ is discrete. More precisely, they consider the calibration
\begin{equation}\label{eq:pullback:factor}
    \vol_{ ( \mathbb{R}^n )^{k} }^{ \times }
    =
    \sum_{ j = 1 }^{ k } \pi^{*}_j \vol_{ \mathbb{R}^n } \in \Lambda^n (\R^n)^k,
\end{equation}
where $\pi_j \colon ( \mathbb{R}^n )^{k} \to \mathbb{R}^n$ is the coordinate projection $(v_1,\ldots, v_k) \mapsto v_j$. The key step in that proof is the following statement.
\begin{proposition}[{\cite[Proposition 7.7]{Hei:Pan:Pry:23}}]\label{prop:7.7:HPP}
Let $n \geq 3$ and $k \geq 1$ and denote $\mathbb{R}^{m} = ( \mathbb{R}^{n} )^{k}$, and fix $\delta > 0$. Then there exists $\epsilon = \epsilon( n,m,\delta) > 0$ for the following. Let $B = B( a, r ) \subset \mathbb{R}^n$ be a ball and let $F \colon B \to \mathbb{R}^m$ be a $(1+\epsilon)$-quasiregular $\vol_{ ( \mathbb{R}^n )^{k} }^{ \times }$-curve. Then
\begin{equation*}
    \sup_{ x \in \rho B } | F(x) - F(a) | \leq \left( \frac{ 1+2\rho}{1-2\rho} + 2^{ \frac{4}{\rho} }\delta \right) \min_{ |x-a| = \rho r } | F(x) - F(a) |
\end{equation*}
for every $0 < \rho \leq 4^{-1}$, where $\rho B \coloneqq B( a, \rho r )$.
\end{proposition}
%The statement uses the short-hand notation $\rho B \coloneqq B( a, \rho r )$.

The proof is based on a similar result for mappings due to Iwaniec \cite{Iwa:87}. Having Propositions \ref{prop:localcontinuity} and \ref{prop:normalfamily} at hand, the proof of \Cref{prop:7.7:HPP} holds as is for inner Möbius rigid calibrations $\omega \in \mathcal{I}_{n}( \R^m )$. Indeed, the proof of \Cref{prop:7.7:HPP} relies on three ingredients. The first is the limit property, \Cref{lemm:confrigid:closed}. The second is the fact that conformal $\omega$-curves are inner Möbius whenever $\omega \in \mathcal{I}_{n}( \R^m )$. The third is the growth estimates for $|G|$ where $G$ is an inner Möbius curve satisfying $|G(0)| < 1$. The latter estimates are exactly the same as for Möbius transformations considered in \cite{Iwa:87}. These facts remain unchanged for any inner Möbius rigid calibration and thus we obtain the following.

\begin{proposition}\label{prop:7.7:HPP-revisited}
Let $m \geq n \geq 3$, $\omega \in \mathcal{I}_{n}( \R^m )$, and $\delta > 0$. Then there exists $\epsilon = \epsilon( \omega,\delta) > 0$ for the following. Let $B = B( a, r ) \subset \mathbb{R}^n$ be a ball and let $F \colon B \to \mathbb{R}^m$ be a $(1+\epsilon)$-quasiregular $\omega$-curve. Then
\begin{equation*}
    \sup_{ x \in \rho B } | F(x) - F(a) | \leq \left( \frac{ 1+2\rho}{1-2\rho} + 2^{ \frac{4}{\rho} }\delta \right) \min_{ |x-a| = \rho r } | F(x) - F(a) |
\end{equation*}
for every $0 < \rho \leq 4^{-1}$.
\end{proposition}
We are now ready to prove \Cref{prop:HPP-revisited}.
\begin{proof}[Proof of \Cref{prop:HPP-revisited}]
Let $H > 1$ and consider $\rho = \rho(H) \in ( 0, 4^{-1} ]$ satisfying
\begin{equation*}
    \frac{ 1+2\rho}{1-2\rho}
    \leq
    \frac{ H + 1 }{ 2 },
\end{equation*}
and $\delta = \delta( H ) \in ( 0, 1 )$ for which
\begin{equation*}
    2^{ \frac{4}{\rho} } \delta \leq \frac{ H - 1 }{ 2 }.
\end{equation*}
Let $\epsilon = \epsilon( \omega, \delta )$ be as in \Cref{prop:7.7:HPP-revisited}, and notice that $\epsilon$ depends only on $\omega$ and $H$. Now consider a nonconstant $(1+\epsilon)$-quasiregular $\omega$-curve $F \colon \Omega \to \mathbb{R}^m$ defined on a domain $\Omega \subset \R^n$. By the above choices of $\delta$, $\rho$ and $\epsilon$, we deduce from \Cref{prop:7.7:HPP-revisited} that $F$ is $( H, \rho )$-weakly quasisymmetric.

It remains to establish the local injectivity of $F$. To this end, let $x_0 \in \Omega$ and consider $r_0 > 0$ so that $B_0 = B( x_0, ( 1 + 2 \rho^{-1} ) r_0 ) \subset \Omega$. We claim that $F$ is constant on $\Omega$ or injective on $B_1 = B( x_0, r_0 )$. Indeed, if two distinct $x, a \in B_1$ satisfy $F(x) = F(a) = v$ for some $v \in \mathbb{R}^m$, we have that $a \in \partial B( x, | x - a | )$ and $B( x, \rho^{-1}|x-a| ) \subset B_0 \subset \Omega$. The $( H, \rho )$-weak quasisymmetry then implies that $F$ is constant on $B( x, |x-a| )$. Similar argument shows that the connected component of $\left\{ y \in \Omega \mid F(y) = v \right\}$ (for the given $v \in \R^m$ above) containing $x$ is relatively open in $\Omega$, so either $F$ is constant or a topological immersion.
\end{proof}

%%%%%%%%%%%%%%%%%%%%%%%%%%%%%%%%%%%%%%%%%%%%%%%%%%%%%%%%%%%%%%%%%%%%%%%%%%%%%%%

We finish with the proof of \Cref{thm:localinjec}.

\weakQS*

\begin{proof}
Let $\omega \in \mathcal{I}_n( \R^m )$ for $3 \leq n \leq m$, $H > 1$ and consider $\rho(H) \in (0, 4^{-1}]$ as in \Cref{prop:7.7:HPP-revisited}. Let $\epsilon = \epsilon( \omega, H ) > 0$ be as in \Cref{prop:HPP-revisited} and let $0 < \delta < \epsilon/(1+\epsilon)$. Consider $\tau \in \mathcal{C}_n( \R^m )$ such that $\| \tau - \omega \|_{ \comass } < \delta$ and a conformal $\tau$-curve $F \colon \Omega \to \R^m$ on a domain $\Omega \subset \R^n$. Then, by Hadamard inequality, we deduce
\begin{equation*}
    \|DF\|^n = \star F^{*}\omega \leq \delta\|DF\|^n + \star F^{*}\tau.
\end{equation*}
Rearranging the inequality implies that $F$ is a $(1-\delta)^{-1}$-quasiregular $\omega$-curve and, by the choice of $\delta > 0$, an $(1+\epsilon)$-quasiregular $\omega$-curve. Then \Cref{prop:7.7:HPP-revisited} gives that $F$ is an $( H, \rho )$-weak quasisymmetry.
\end{proof}

\bibliographystyle{alpha}

\end{document}